\documentclass[10pt,letterpaper]{article}
\makeatletter
\usepackage{pdfsync}
\usepackage{cite}
\usepackage{amscd}
\usepackage{amsmath}
\usepackage{latexsym}
\usepackage{amsfonts}
\usepackage{amssymb}
\usepackage{amsthm}
\usepackage{graphicx}
\usepackage{indentfirst}
\usepackage{enumerate}
\usepackage{amstext}
\usepackage[dvipdfm,
            pdfstartview=FitH,
            CJKbookmarks=true,
            bookmarksnumbered=true,
            bookmarksopen=true,
            colorlinks=true, 
            pdfborder=001,   
            citecolor=magenta,
            linkcolor=blue,
            ]{hyperref}       

\newtheorem{thm}{\textbf Theorem}[section]
\newtheorem{lem}{\textbf Lemma}[section]
\newtheorem{rem}{\textbf Remark}[section]

\newtheorem{prop}{\textbf Proposition}[section]

\numberwithin{equation}{section}

\newcommand{\be}{\begin{eqnarray}}
\newcommand{\ee}{\end{eqnarray}}

\newcommand{\bes}{\begin{eqnarray*}}
\newcommand{\ees}{\end{eqnarray*}}

\begin{document}
\begin{titlepage}
\title{\bf Navier-Stokes equations with external forces in
Besov-Morrey spaces }

\author{Boling Guo$^{~a}$,  \quad Guoquan Qin$~^{b,*}$
\\[10pt]
\small {$^a $ Institute of Applied Physics and Computational Mathematics,
China Academy of Engineering Physics,}\\
\small {   Beijing,  100088,  P. R. China}\\[5pt]
\small {$^b $ Graduate School of China Academy of Engineering Physics,}\\
\small {  Beijing,  100088,  P. R. China}\\[5pt]
}
\footnotetext
{*~Corresponding author.

~~~~~E-mail addresses: gbl@iapcm.ac.cn(B. Guo), 690650952@qq.com(G. Qin).}


\date{}
\end{titlepage}
\maketitle
\begin{abstract}
We establish the  existence and uniqueness  of local strong solutions to the Navier-Stokes
equations with arbitrary initial data and external forces
in the homogeneous Besov-Morrey space.
The local solutions
can be extended globally in time
provided the initial data and external forces are small.
We adapt the method introduced in \cite{ks6},
where the Besov space is considered, to
the
setting of the homogeneous Besov-Morrey space.



\end{abstract}

~~\noindent{ \textbf{Key words}: Navier-Stokes equation, Besov-Morrey space,
maximal Lorentz regularity. }



\section{Introduction}
\setcounter{equation}{0}
In this paper, we consider the following incompressible Navier-Stokes
equations in $\mathbb{R}^{n}(n\geq 2):$
\begin{eqnarray}\label{ns}
\begin{cases}
u_{t}-\Delta u+u\cdot\nabla u+\nabla p=f, \ \ x\in \mathbb{R}^{n}, \ \ t\in(0, T),\\
\mbox{div}u=0,\ \ x\in \mathbb{R}^{n}, \ \ t\in(0, T),\\
u|_{t=0}=a,\ \ x\in \mathbb{R}^{n},
\end{cases}
\end{eqnarray}
where $u=u(x, t)=(u_{1}(x, t),..., u_{n}(x, t))$ and $p=p(x, t)$ denote the unknown velocity
vector and the unknown pressure at the point $x=(x_{1}, x_{2}, ..., x_{n})\in \mathbb{R}^{n}$
and the time $t\in(0, T),$ respectively, while $a=a(x)=(a_{1}(x),..., a_{n}(x))$
and $f=f(x, t)=(f_{1}(x, t),..., f_{n}(x, t))$ are the given initial velocity  vector and the
external force,  respectively.

As we all know,   equation (\ref{ns}) is
invariant under the following  change of scaling:
\begin{equation*}
 u_{\lambda}(x, t)=\lambda u(\lambda x, \lambda^{2} t)\ \
 \text{and}\ \
p_{\lambda}(x, t)=\lambda^{2} p(\lambda x, \lambda^{2} t)\ \
\text{for all}\ \ \lambda>0.
\end{equation*}

If  a Banach space $\mathcal{Y}$
satisfies $\|u_{\lambda}\|_{\mathcal{Y}}=\|u\|_{\mathcal{Y}}$
for all $\lambda>0$, then it is called scaling invariant to (\ref{ns}).
For
instance, in the setting of  Lebesgue spaces $L^{p}(\mathbb{R}^{n})$,
we see that the scaling invariant space $\mathcal{Y}$
to  (\ref{ns})
is the so-called  Serrin class $L^{\alpha}(0, \infty ; L^{p}(\mathbb{R}^{n}))$
for $2 / \alpha+n / p=1$ with $n \leq p \leq \infty$.
For the initial data $a$
and the external force $f$,
the
corresponding scaling laws
 are like $a_{\lambda}(x)=\lambda a(\lambda x)$
and $f_{\lambda}(x, t)=\lambda^{3} f(\lambda x, \lambda^{2} t),$
respectively.
Therefore, it is suitable to solve (\ref{ns}) in Banach spaces $X$
for $a$ and $Y$ for $f$  with the properties that
$\|a_{\lambda}\|_{X}=\|a\|_{X}$ and $\|f_{\lambda}\|_{Y}=\|f\|_{Y}$
for all $\lambda>0$, respectively.

Let's first recall some results with respect to  the space $X.$
Since the
pioneer work of Fujita-Kato \cite{fk}, many efforts have been made to find such a space $X$ as large
as possible.
For instance,
 Kato \cite{k2} and Giga-Miyakawa \cite{gm}
 succeed to find the space $X=L^{n}(\mathbb{R}^{n}).$
 Kozono-Yamazaki \cite{ky2} and Cannone-Planchon \cite{cp1}
 extended $X=L^{n}(\mathbb{R}^{n})$  to $X=L^{n, \infty}(\mathbb{R}^{n})$
 and $X=\dot{B}_{p, \infty}^{-1+n / p}(\mathbb{R}^{n})$ with $n<p<\infty,$
 where $L^{n, \infty}(\mathbb{R}^{n})$ is the weak Lebesgue spaces and
 $\dot{B}_{p, q}^{s}(\mathbb{R}^{n})$ denotes the homogeneous Besov space.
The largest space of $X$ was obtained by Koch-Tataru \cite{kt}
who proved local well-posedness of  (\ref{ns})
for $a\in X=BMO^{-1}=\dot{F}^{-1}_{\infty, 2}(\mathbb{R}^{n}),$
where  $\dot{F}^{s}_{q, r}(\mathbb{R}^{n})$  denotes the homogeneous Triebel-Lizorkin space.
The result in \cite{kt} seems to be optimal in the sense that continuous
dependence of solutions with respect to the initial data breaks down in
$X=\dot{B}^{-1}_{\infty, r}(\mathbb{R}^{n})$ for $2< r\leq \infty,$
which was proved by  Bourgain-Pavlovi\'{c} \cite{bp},
Yoneda \cite{y} and Wang \cite{w}.
 Amann \cite{a2} has established a systematic treatment
of strong solutions in various function spaces such as Lebesgue space $L^{p}(\Omega)$, Bessel potential
space $H^{s, p}(\Omega)$, Besov space $B_{p, q}^{s}(\Omega)$ and Nikol'skii space
$N^{s, p}(\Omega)$ in general domains $\Omega.$
In
this direction, based on the Littlewood-Paley decomposition, Cannone-Meyer \cite{cm} showed how to
choose the Banach spaces $X$ for $a$ and $\mathcal{Y}$ for $u$. Besides these results, in terms of the Stokes
operator, Farwig-Sohr \cite{fs1} and Farwig-Sohr-Varnhorn \cite{fsv} proved a necessary and sufficient condition
on $a$ such that weak solutions $u$ belong to the Serrin class.

On the other hand, in comparison with a number of papers on well-posedness with respect to the initial data, there is a little
 literature for investigating the suitable space $Y$ of external forces
 $f$ satisfying
 $\|f_{\lambda}\|_{Y}=\|f\|_{Y}$  for all $\lambda>0.$
 For instance, Giga-Miyakawa \cite{gm} proved existence
of strong solutions for
\begin{equation}
Y=\left\{f \in C\left(0, \infty ; L^{n}\left(\mathbb{R}^{n}\right)\right)
\|f\|_{Y}=\sup _{0<t<\infty} t^{1-\delta}\left\|(-\Delta)^{-\delta}\mathbb{P} f(t)\right\|_{L^{n}}<\infty\right\}
\end{equation}
for some $\delta>0$,
where $\mathbb{P}$ denotes the Helmholtz projection.
Cannone-Planchon \cite{cp2} treated
the case $n = 3$  and showed that
\begin{equation}
Y=\left\{f=\operatorname{div} F ; F \in L^{\alpha}(0, \infty ; L^{p}(\mathbb{R}^{3}))\right\}
\end{equation}
for $2 / \alpha+3 / p=2$ with  $2 / 3<p<\infty$  is a suitable space.
See also Planchon \cite{p}.
After introducing the space of pseudo
measures $\mathcal{P} \mathcal{M}^{k}=\left\{a \in \mathcal{S}^{\prime} ; \sup _{\xi \in \mathbb{R}^{n}}|\xi|^{k}|\hat{a}(\xi)|<\infty\right\}$,
Cannone-Karch \cite{ck} showed that the pair
of $X=\mathcal{P} \mathcal{M}^{2}$ and $Y=C_{w}\left((0, \infty) ; \mathcal{P} \mathcal{M}^{0}\right)$ is suitable for $n = 3$.
Recently,  Kozono-Shimizu \cite{ks2}
constructed mild solutions for $X=L^{n, \infty}\left(\mathbb{R}^{n}\right)$
and $Y=L^{\alpha, \infty}\left(0, \infty ; L^{p, \infty}\left(\mathbb{R}^{n}\right)\right)$,  where
$2 / \alpha+n / p=3$ and $\max\{1, n/3\} < p < \infty.$
Another choice of $X$ and $Y$ was obtained by Kozono-Shimizu \cite{ks3} which proved
the existence of mild solutions in the case when
\begin{equation*}
X=\dot{B}_{p, \infty}^{-1+\frac{n}{p}}(\mathbb{R}^{n})
\end{equation*}
and
\begin{equation*}
Y=\left\{f \in C_{w}(0, \infty; \dot{B}_{p_{0}, \infty}^{s_{0}}(\mathbb{R}^{n})); \|f\|_{Y} \equiv \sup _{0<t<\infty} t^{\frac{n}{2}(\frac{3}{n}-\frac{1}{p_{0}})+\frac{s_{0}}{2}}\|f(t)\| _{\dot{B}_{p_{0}, \infty}^{s_{0}}}<\infty\right\}
\end{equation*}
for $n < p < \infty$, $ n/3 < p_{0} \leq p$ and $s_{0} <\min\{0, n/p_{0}-1\}$.
A  similar investigation
in the time-weighted Besov spaces were given by Kashiwagi \cite{k1} and Nakamura \cite{n}.

Although these spaces $X$ and $Y$ are scaling invariant spaces, the solutions
$u$ are  just  mild solutions  in most cases.
To
gain enough regularity for the validity of (\ref{ns}),
  the H\"{o}lder continuity of $f(t) \in L^{n}(\mathbb{R}^{n})$
  is assumed
  on $0<t<T$.
Namely,
although $u$ belongs to the scaling invariant class $\mathcal{Y}$,
  whether
$-\Delta u$ and $\partial_{t} u$
are well-defined in some $L^{p}$-space is not exactly sure.
In this
direction, the result given by \cite{ks3} may be regarded as an almost optimal theorem on gain of
regularity of the mild solution $u$ since it holds that $u(t) \in H^{s}(\mathbb{R}^{n})$
for all $s < 2$ and for almost
everywhere $t\in (0,\infty).$

Very recently,
 based on the maximal Lorentz regularity theorem on the Stokes
equations in $\dot{B}^{s}_{p,q}(\mathbb{R}^{n})$,
Kozono-Shimizu \cite{ks6}
proved the existence of
 a strong solution $u$ with $-\Delta u$ and $\partial_{t} u$
in some Besov space such that (\ref{ns}) is fulfilled
almost everywhere in $\mathbb{R}^{n} \times(0, T)$
 for a suitable choice of spaces $X$ and $Y$.
Their obtained  solution  also belongs  to the usual Serrin class
$L^{\alpha_{0}}(0, T;L^{p_{0}}(\mathbb{R}^{n}))$ for
 $2/\alpha_{0} +n/p_{0} = 1$ with
$n <p_{0} < \infty$.
Also, they make it clear how to relate $X$ and $Y$ to
$\mathcal{Y}$.
 Their  result covers almost all previous results like \cite{k2}, \cite{ky1},
 \cite{c} and \cite{cp1}.

Following  the spirit of \cite{ks6} and \cite{ks2},
we try to find new  suitable spaces for  $X$ and $Y.$
  And the purpose of this paper is to
  prove the existence and uniqueness of local strong
solutions to equation (\ref{ns}) for arbitrary initial data
$a\in\dot{\mathcal{N}}_{p,\mu,r}^{-1+(n-\mu)/p}$
and external forces $f\in L^{\alpha, r}(0, T; \dot{\mathcal{N}}_{q,\mu,\infty}^{s})$,
where   $2/\alpha+(n-\mu)/q-s=3$,
 $1< p\leq q, 0\leq \mu< n$, $s>-1$ and $0<T\leq\infty$.
Here $\dot{\mathcal{N}}_{p,\mu,r}^{s}$
 denotes the homogeneous Besov-Morrey space.
We also show the existence and uniqueness of global solutions
for small $a \in \dot{\mathcal{N}}_{p,\mu, q}^{-1+\frac{n-\mu}{p}}(\mathbb{R}^{n})$
and small $f \in L^{\alpha, r}((0, \infty) ; \dot{\mathcal{N}}_{q,\mu, \infty}^{s}(\mathbb{R}^{n}))$.
The method is essentially adapted from \cite{ks6}
and we prove the maximal Lorentz regularity theorem
for the Stokes equation in the setting of homogeneous Besov-Morrey space.
When $\mu=0,$ one can obtain
$\dot{\mathcal{N}}_{p,\mu,r}^{-1+(n-\mu)/p}=\dot{B}_{p,r}^{-1+n/p}$
and we thus generalize the existence result in \cite{ks6}.

To state our results,
let us first denote by $\mathbb{P}$ the Helmholtz projection from
the Lebesgue space $L^{p}(1<p<\infty)$ onto
the subspace of solenoidal vector fields
$\mathbb{P}L^{p}=L^{p}_{\sigma}=\{f\in L^{p}; \mbox{div}f=0\}$ as a bounded operator.
It is well known that $\mathbb{P}$ is expressed as
\begin{equation*}
 \mathbb{P}=(\mathbb{P}_{ij})_{1\leq i,j\leq n},\ \
 \mathbb{P}_{ij}=\delta_{ij}+R_{i}R_{j}, \ \
 i,j=1,2,..., n,
\end{equation*}
where $\{\delta_{ij}\}_{1\leq i,j\leq n}$ is the Kronecker symbol and
$R_{i}=\partial_{i}(-\Delta)^{-1/2}(i=1,2,..., n)$
are the Riesz transforms.
From the Calder\'{o}n-Zygmund operator theory, for
$1<p<\infty, 0 \leq \mu< n$, the boundedness
of Riesz transform $R_{j}$ on the Morrey space
  $\mathcal{M}_{p,\mu}$(see Section 2 for the definition)
is established in [\cite{t}, Proposition 3.3]
and hence $\mathbb{P}$ is  bounded on $\mathcal{M}_{p,\mu}$.

 The original equations (\ref{ns}) can be rewritten as the following abstract equation:
\begin{eqnarray}\label{eq101}
\begin{cases}
  \partial_{t}u+Au+\mathbb{P}(u\cdot\nabla u)=\mathbb{P}f \ \ on \ \ (0, T),\\
  u(0)=a,
  \end{cases}
\end{eqnarray}
where $A=-\mathbb{P}\Delta$ is the Stokes operator.
The solution $u$ of (\ref{eq101}) is called a strong solution of (\ref{ns}).

Our first result is  the maximal Lorentz regularity theorem of the Stokes equations in
the homogenous Besov-Morrey space
$\dot{\mathcal{N}}_{q,\mu,\beta}^{s}$(see Section 2 for the definition).

\begin{thm}\label{thm1}
Let $1<p\leq q<\infty, 1<\alpha<\infty, 1\leq\beta\leq\infty,
1\leq r\leq\infty, 0\leq \mu< n$
and $s\in\mathbb{R}.$ Assume that
\begin{equation}\label{eqmax102}
\frac{n-\mu}{q}\leq\frac{n-\mu}{p}<\frac{2}{\alpha}+\frac{n-\mu}{q}.
\end{equation}
For  every $a \in \dot{\mathcal{N}}_{p, \mu, r}^{k}$ with
 $k=s+2+[(n-\mu)/p-(n-\mu)/q]-2/\alpha$ and every $f \in L^{\alpha, r}(0, T ; \dot{\mathcal{N}}_{q,\mu, \beta}^{s})$
with $0<T \leq \infty,$ there exists a unique solution $u$ of
\begin{eqnarray}\label{eqmaxs}
\left\{
\begin{array}{c}{\frac{d u}{d t}+A u=\mathbb{P} f \quad \text { a.e. } t \in(0, T) \text { in } \dot{\mathcal{N}}_{q,\mu, \beta}^{s}} \\
 {u(0)=a \quad \text { in } \dot{\mathcal{N}}_{p,\mu, r}^{k}}
 \end{array}\right.
\end{eqnarray}
in the class
\begin{equation}
u_{t}, A u \in L^{\alpha, r}(0, T ; \dot{\mathcal{N}}_{q,\mu, \beta}^{s}).\nonumber
\end{equation}
Moreover, such a solution $u$ is subject to the estimate
\begin{equation}\label{eqmax103}
\|u_{t}\|_{L^{\alpha, r}(0, T ; \dot{\mathcal{N}}_{q,\mu, \beta}^{s})}+\|A u\|_{L^{\alpha, r}(0, T ; \dot{\mathcal{N}}_{q, \mu, \beta}^{s})} \leq C\left(\|a\|_{\dot{\mathcal{N}}_{p,\mu, r}^{k}}+\|f\|_{L^{\alpha, r}(0, T ; \dot{\mathcal{N}}_{q,\mu, \beta}^{s})}\right),
\end{equation}
where $C=C(n,\mu, q, \alpha, r, \beta, s, p)$ is a constant independent of $0<T \leq \infty$.
\end{thm}

Using Theorem \ref{thm1}, we  establish the existence and
uniqueness of local strong solutions to (\ref{ns}) for arbitrary large initial
data $a$ and large external force $f$ in the setting of homogenous
Besov-Morrey space. Our second results now reads:
\begin{thm}\label{thm2}
Let $1<q<\infty, 1<\alpha<\infty, 0\leq \mu<n$ and $s>-1$
 satisfy $2 / \alpha+(n-\mu) / q-s=3$.
Let $1 \leq r \leq \infty$.
Assume that $1 \leq p \leq q$ satisfies $(\ref{eqmax102}).$\\
(i)\ \ In case $1 \leq r<\infty.$
For every $a \in \dot{\mathcal{N}}_{p,\mu, r}^{-1+(n-\mu) / p}$ and $f \in L^{\alpha, r}(0, T ; \dot{\mathcal{N}}_{q,\mu, \infty}^{s})$ with $0<T \leq \infty$,
there exists $0<T_{*} \leq T$ and a unique solution $u$ on $\left(0, T_{*}\right)$ of
\begin{equation}\label{eqmaxns}
\left\{\begin{array}{l}{\frac{d u}{d t}+A u+\mathbb{P}(u \cdot \nabla u)=\mathbb{P} f \quad \text { a.e. } t \in\left(0, T_{*}\right) \text { in } \dot{\mathcal{N}}_{q,\mu, \infty}^{s}} \\ {u(0)=a \quad \text { in } \dot{\mathcal{N}}_{p,,\mu, r}^{-1+(n-\mu)/ p}}\end{array}\right.
\end{equation}
in the class
\begin{equation}
u_{t}, A u \in L^{\alpha, r}(0, T_{*} ; \dot{\mathcal{N}}_{q,\mu,  \infty}^{s}).
\end{equation}
Moreover, such a solution $u$ satisfies that
\begin{equation}\label{eqmax105}
u \in L^{\alpha_{0}, r}(0, T_{*} ; \dot{\mathcal{N}}_{q_{0},\mu, 1}^{s_{0}}) \quad \text { for } 2 / \alpha_{0}+(n-\mu) / q_{0}-s_{0}=1
\end{equation}
with $q \leq q_{0}, \alpha<\alpha_{0} $ and $\max \{s, (n-\mu) / p-1\}<s_{0}$.\\
(ii)\ \ In case $r=\infty$.
In addition to $(\ref{eqmax102}),$ assume that $n-\mu<p \leq q<\infty.$  There is a constant
$\eta=\eta(n,\mu, q, \alpha, s, p)$ such that if $a \in
 \dot{\mathcal{N}}_{p, \mu,\infty}^{-1+(n-\mu) / p}$ and $f \in L^{\alpha, \infty}(0, T ; \dot{\mathcal{N}}_{q,\mu, \infty}^{s})$ satisfy
\begin{equation}\label{eqmax106}
\sup _{N \leq j<\infty} 2^{(-1+\frac{n-\mu}{p}) j}\left\|\varphi_{j} * a\right\|_{\mathcal{M}_{p,\mu}}+\limsup_{t \rightarrow \infty} t \left|\{\tau \in(0, T) ;\|f(\tau)\|_{\dot{\mathcal{N}}_{q,\mu, \infty}^{s}}>t\}\right|^{1 / \alpha} \leq \eta
\end{equation}
for some $N \in \mathbb{Z},$ then there exist $0<T_{*} \leq T$
and a solution $u$ of $(\ref{eqmaxns})$ on $\left(0, T_{*}\right)$ in the class
\begin{equation}\label{eqmax107}
u_{t}, A u \in L^{\alpha, \infty}(0, T_{*} ; \dot{\mathcal{N}}_{q,\mu,  \infty}^{s}).
\end{equation}
Moreover, such a solution  $u$ satisfies that
\begin{equation}\label{eqmax108}
u \in L^{\alpha_{0}, \infty}(0, T_{*} ; \dot{\mathcal{N}}_{q_{0},\mu, 1}^{s_{0}}) \quad \text { for } 2 / \alpha_{0}+(n-\mu) / q_{0}-s_{0}=1
\end{equation}
with $q \leq q_{0}, \alpha<\alpha_{0}$ and $\max \{s, (n-\mu) / p-1\}<s_{0}$.

Concerning the uniqueness, there is a constant $\kappa=\kappa(n,\mu, q, \alpha, s, p)>0$ such that if the initial data $a$ and the solution $u$ of $(\ref{eqmaxns})$ on $(0, T_{*})$ in the class $(\ref{eqmax107})$ satisfy
\begin{eqnarray}\label{eqmax109}
&&\sup _{N \leq j<\infty} 2^{(-1+\frac{n-\mu}{p}) j}\|\varphi_{j} * a\|_{\mathcal{M}_{p,\mu}}\nonumber\\
&&+\limsup _{t \rightarrow \infty} t \bigg|\{\tau \in(0, T_{*}) ;\|u_{\tau}(\tau)\|_{\dot{\mathcal{N}}_{{q,\mu, \infty}}^{s}}
+\|A u(\tau)\|_{\dot{\mathcal{N}}_{q,\mu, \infty}^{s}}>t\}\bigg|^{\frac{1}{\alpha}} \leq \kappa
\end{eqnarray}
for some $N \in \mathbb{N},$ then $u$ is unique.
\end{thm}

Our third result is  the global existence and uniqueness of strong solutions to (\ref{ns}) for small $a$ and $f.$

\begin{thm}\label{thm3}
Let $1<q<\infty, 1<\alpha<\infty, 0\leq\mu<n$ and $s>-1$ satisfy $2 / \alpha+(n-\mu )/ q-s=3.$ Let $1 \leq r \leq \infty$.
Assume that $1 \leq p \leq q$ satisfies $(\ref{eqmax102}) .$ There exists a constant $\varepsilon_{*}=\varepsilon_{*}(n, \mu, q, \alpha, s, p, r)>0$ such that
if $a \in \dot{\mathcal{N}}_{p,\mu, r}^{-1+(n-\mu) / p}$ and $f \in L^{\alpha, r}(0, \infty ; \dot{\mathcal{N}}_{q,\mu, \infty}^{s})$ satisfy
\begin{equation}\label{eqmax111}
\|a\|_{ \dot{\mathcal{N}}_{p,\mu, r}^{-1+(n-\mu) / p}}+\|f\|_{L^{\alpha, r}(0, \infty ; \dot{\mathcal{N}}_{q,\mu, \infty}^{s})} \leq \varepsilon_{*},
\end{equation}
then there is a solution $u$ of $(\ref{eqmaxns})$ on $(0, \infty)$ in the class
\begin{equation}\label{eqmax112}
u_{t}, A u \in L^{\alpha, r}(0, \infty ; \dot{\mathcal{N}}_{q,\mu, \infty}^{s}).
\end{equation}
Moreover, such a solution $u$ satisfies that
\begin{equation}\label{eqmax113}
u \in L^{\alpha_{0}, r}(0, \infty ; \dot{\mathcal{N}}_{q_{0},\mu, 1}^{s_{0}}) \quad \text { for } 2 / \alpha_{0}+(n-\mu) / q_{0}-s_{0}=1
\end{equation}
with $q \leq q_{0}, \alpha<\alpha_{0}$ and $\max \{s, (n-\mu) / p-1\}<s_{0}$.

In the case $1 \leq r<\infty,$ the solution $u$ of $(\ref{eqmaxns})$ on
$(0, \infty)$ in the class $(\ref{eqmax112})$ is unique.
In the case $ r=\infty,$
the uniqueness holds under the same condition as $(\ref{eqmax109})$ with $T_{*}$ replaced by $\infty$
provided $n-\mu<p \leq q.$
\end{thm}

 For the global strong solution $u$ of (\ref{ns})
obtained  in Theorem \ref{thm3},
if the initial data $a$ and
the external force $f$ has certain additional regularity,
then the strong solution $u$ has the corresponding regularity to that of $a$ and $f$. This can be stated in the following Theorem:

\begin{thm}\label{thm4}
Let $1 \leq p \leq q<\infty, 1<\alpha<\infty, 0\leq \mu<n$ and $-1<s$ be as in Theorem \ref{thm3} and let $1 \leq r \leq \infty$.
Let $1<\alpha^{*} \leq \alpha,$ $0\leq \mu^{*}\leq\mu<n$ and $1<p^{*} \leq q^{*}<\infty$ satisfy $(n-\mu^{*}) / p^{*}<2 / \alpha^{*}+(n-\mu^{*})/ q^{*},$ and let $-1<s^{*} \leq s$ and
$1 \leq r^{*} \leq \infty .$ There is a positive constant $\varepsilon_{*}^{\prime}=\varepsilon_{*}^{\prime}(n,\mu, \mu^{*}, q, \alpha, s, p, r, q^{*}, \alpha^{*}, s^{*}, p^{*}, r^{*}) \leq \varepsilon_{*}$ with the
same $\varepsilon_{*}$ as in Theorem \ref{thm3} such that if $a \in \dot{\mathcal{N}}_{p,\mu, r}^{-1+(n-\mu) / p} \cap \dot{\mathcal{N}}_{p^{*},\mu^{*},  r^{*}}^{k^{*}} $ with $ k^{*}=2+(n-\mu^{*}) / p^{*}-(2 / \alpha^{*}+(n-\mu^{*})/ q^{*}-s^{*})$
and $f \in L^{\alpha, r}(0, \infty ; \dot{\mathcal{N}}_{q,\mu, \infty}^{s}) \cap L^{\alpha^{*}, r^{*}}(0, \infty ; \dot{\mathcal{N}}_{q^{*},\mu^{*}, \infty}^{s^{*}})$ satisfy

\begin{equation}\label{eqmax126}
  \|a\|_{\dot{\mathcal{N}}_{p,\mu, r}^{-1+(n-\mu) / p}}+\|f\|_{L^{\alpha, r}(0, \infty ; \dot{\mathcal{N}}_{q,\mu, \infty}^{s})} \leq \varepsilon_{*}^{\prime},
\end{equation}
then the solution $u$ of (\ref{eqmaxns}) given by Theorem \ref{thm3} has the additional property that
\begin{equation}
u_{t}, A u \in L^{\alpha^{*}, r^{*}}(0, \infty ; \dot{\mathcal{N}}_{q^{*},\mu^{*}, \infty}^{s^{*}}).
\end{equation}
\end{thm}

This paper is organized as follows.
In section 2,
we collect some useful lemmas
and prove
the  maximal Lorentz regularity theorem
for  the Stokes equations
 in the setting of   the homogeneous
Besov-Morrey space.
 In Section 3, we present
various estimates for
the nonlinear term  in the homogeneous Besov-Morrey space.
 Section 4 is devoted to the proof of
  the existence and uniqueness theorem of local strong solutions to (\ref{ns})
for arbitrary large data $a$ and $f$.
The global strong solutions for small $a$ and $f$ are also discussed.

\section{Preliminaries}

In this section, we first collect some well-known properties about
Sobolev-Morrey and Besov-Morrey spaces,
then we prove Theorem \ref{thm1}.

\subsection{Besov-Morrey space}

The basic properties of Morrey and Besov-Morrey space is reviewed in the present
subsection for the reader's convenience,
more details can be found in \cite{ap,hw12,hw17,hw25,t}.

Let $Q_{r}(x_{0})$  be the open ball in $\mathbb{R}^{n}$ centered at $x_{0}$ and with radius $r>0.$
Given two parameters $1\leq p<\infty$ and $0\leq \mu<n,$
the Morrey spaces $\mathcal{M}_{p, \mu}=\mathcal{M}_{p, \mu}(\mathbb{R}^{n})$ is defined to be
the set of functions $f\in L^{p}(Q_{r}(x_{0}))$ such that
\begin{equation}\label{extraeq201}
  \|f\|_{p,\mu}\triangleq
  \sup_{x_{0}\in \mathbb{R}^{n}}\sup_{r>0}r^{-\mu/p}\|f\|_{L^{p}(Q_{r}(x_{0}))}<\infty
\end{equation}
which is a Banach space endowed with norm (\ref{extraeq201}).
For $s\in \mathbb{R}$ and $1\leq p<\infty,$ the homogenous Sobolev-Morrey space
 $\mathcal{M}_{p,\mu}^{s}=(-\Delta)^{-s/2}\mathcal{M}_{p, \mu}$ is the Banach space
 with norm
 \begin{equation}\label{eq202}
  \|f\|_{\mathcal{M}_{p,\mu}^{s}}=\|(-\Delta)^{s/2}f\|_{p,\mu}.\nonumber
 \end{equation}

Taking $p=1,$ we have $\|f\|_{L^{1}\left(Q_{r}\left(x_{0}\right)\right)}$
denotes the total variation of $f$ on open ball  $Q_{r}\left(x_{0}\right)$
and $\mathcal{M}_{1, \mu}$
stands for space of signed measures. In particular, $\mathcal{M}_{1, 0}=\mathcal{M}$
 is  the space of finite measures. For $p>1,$ we have $\mathcal{M}_{p, 0}=L^{p}$
  and $\mathcal{M}_{p, 0}^{s}=\dot{H}_{p}^{s}$ is the well known Sobolev space.
  The space $L^{\infty}$ corresponds to $\mathcal{M}_{\infty, \mu}.$
  Morrey and Sobolev-Morrey spaces present the following scaling
\begin{equation}\label{eq203}
\|f(\lambda \cdot)\|_{p, \mu}=\lambda^{-\frac{n-\mu}{p}}\|f\|_{p, \mu}\nonumber
\end{equation}
and
\begin{equation}\label{eq204}
\|f(\lambda \cdot)\|_{\mathcal{M}_{p, \mu}^{s}}=\lambda^{s-\frac{n-\mu}{p}}\|f\|_{\mathcal{M}_{p, \mu}^{s}},\nonumber
\end{equation}
where the  exponent $s-\frac{n-\mu}{p}$ is called scaling index and $s$ is called regularity index.
We have that
\begin{equation}\label{eq205}
(-\Delta)^{l / 2} \mathcal{M}_{p, \mu}^{s}=\mathcal{M}_{p, \mu}^{s-l}.\nonumber
\end{equation}
Morrey spaces contain Lebesgue and weak-$L^{p}$, with the same scaling index. Precisely, we have
the continuous proper inclusions
\begin{equation}\label{eq206}
L^{p}\left(\mathbb{R}^{n}\right) \nsubseteq \operatorname{weak}-L^{p}\left(\mathbb{R}^{n}\right) \nsubseteq \mathcal{M}_{r, \mu}\left(\mathbb{R}^{n}\right),\nonumber
\end{equation}
where $r<p$ and $\mu=n(1-r/p)$(see e.g. \cite{hw23}).

Let $\mathcal{S}(\mathbb{R}^{n})$ and $\mathcal{S}^{'}(\mathbb{R}^{n})$ be the Schwartz space
and the tempered distributions, respectively. Let $\varphi\in \mathcal{S}(\mathbb{R}^{n})$
be nonnegative radial function such that
\begin{equation}
\operatorname{supp}(\varphi) \subset\left\{\xi \in \mathbb{R}^{n} ; \frac{1}{2}<|\xi|<2\right\}\nonumber
\end{equation}
and
\begin{equation}
\sum_{j=-\infty}^{\infty} \varphi_{j}(\xi)=1, \text { for all } \xi \neq 0,\nonumber
\end{equation}
where $\varphi_{j}(\xi)=\varphi(2^{-j}\xi)$.
Let $\phi(x)=\mathcal{F}^{-1}(\varphi)(x)$
and $\phi_{j}(x)=\mathcal{F}^{-1}\left(\varphi_{j}\right)(x)=2^{j n} \phi\left(2^{j} x\right)$
where $\mathcal{F}^{-1}$ stands for inverse Fourier transform.
For $1 \leq q<\infty, 0 \leq \mu<n$ and $s \in \mathbb{R}$,
the homogeneous Besov-Morrey space
$\dot{\mathcal{N}}_{q, \mu, r}^{s}\left(\mathbb{R}^{n}\right)$
 ($\dot{\mathcal{N}}_{q, \mu, r}^{s}$ for short)
is defined to be the set of $u\in\mathcal{S}^{'}(\mathbb{R}^{n}) $,
 modulo polynomials $\mathcal{P},$
 such that
 $\mathcal{F}^{-1} \varphi_{j}(\xi) \mathcal{F} u \in \mathcal{M}_{q, \mu}$ for all $j \in \mathbb{Z}$ and
\begin{equation}\label{eq207}
\|u\|_{\dot{\mathcal{N}}_{q, \mu, r}^{s}}=\left\{\begin{array}{ll}{\left(\sum_{j \in \mathbb{Z}}\left(2^{j s}\left\|\phi_{j} * u\right\|_{q, \mu}\right)^{r}\right)^{\frac{1}{r}}<\infty,} & {1 \leq r<\infty} \\ {\sup _{j \in \mathbb{Z}} 2^{j s}\left\|\phi_{j} * u\right\|_{q, \mu}<\infty,} & {r=\infty}\end{array}\right.\nonumber
\end{equation}
The space $\dot{\mathcal{N}}_{q, \mu, r}^{s}\left(\mathbb{R}^{n}\right)$
is a Banach space and, in particular, $\dot{\mathcal{N}}_{q, 0, r}^{s}=\dot{B}_{q, r}^{s}$ (case $\mu=0$ )
corresponds to the
homogeneous Besov space. We have the real-interpolation properties
\begin{equation}\label{hweq1}
\dot{\mathcal{N}}_{q, \mu, r}^{s}=\left(\mathcal{M}_{q, \mu}^{s_{1}}, \mathcal{M}_{q, \mu}^{s_{2}}\right)_{\theta, r}
\end{equation}
and
\begin{equation}\label{eq208}
\dot{\mathcal{N}}_{q, \mu, r}^{s}=\left(\dot{\mathcal{N}}_{q, \mu, r_{1}}^{s_{1}}, \dot{\mathcal{N}}_{q, \mu, r_{2}}^{s_{2}}\right)_{\theta, r}
\end{equation}
for all $s_{1} \neq s_{2}, 0<\theta<1$ and $s=(1-\theta) s_{1}+\theta s_{2}.$
Here $(X, Y)_{\theta, r}$ stands for the real interpolation space
between $X$ and $Y$  constructed via the $K_{\theta, q}-$method.
Recall that $(\cdot, \cdot)_{\theta, r}$ is an exact interpolation functor of
exponent $\theta$ on the category of normed spaces.

In the next lemmas, we collect basic facts about Morrey spaces and Besov-Morrey spaces
(see \cite{ap, hw12, t}).
\begin{lem}\label{lem1}
Suppose that $s_{1}, s_{2} \in \mathbb{R}, 1 \leq p_{1}, p_{2}, p_{3}<\infty$ and $0 \leq \mu_{i}<n, i=1,2,3$.\\
(i)(Inclusion)\ \ If $\frac{n-\mu_{1}}{p_{1}}=\frac{n-\mu_{2}}{p_{2}}$ and $p_{2} \leq p_{1}$,
then
\begin{equation*}\label{eq209}
\mathcal{M}_{p_{1}, \mu_{1}} \hookrightarrow \mathcal{M}_{p_{2}, \mu_{2}} \quad \text { and } \dot{\mathcal{N}}_{p_{1}, \mu_{1}, 1}^{0} \hookrightarrow \mathcal{M}_{p_{1}, \mu_{1}} \hookrightarrow \dot{\mathcal{N}}_{p_{1}, \mu_{1}, \infty}^{0}.
\end{equation*}
(ii)(Sobolev-type embedding)\ \ Let
$j=1,2$ and $p_{j}, s_{j}$ be $p_{2} \leq p_{1}, s_{1} \leq s_{2}$ such that $s_{2}-\frac{n-\mu_{2}}{p_{2}}=s_{1}-\frac{n-\mu_{1}}{p_{1}}$,
then we have
\begin{equation*}\label{eq210}
\mathcal{M}_{p_{2}, \mu}^{s_{2}} \hookrightarrow \mathcal{M}_{p_{1}, \mu}^{s_{1}},\left(\mu=\mu_{1}=\mu_{2}\right)
\end{equation*}
and for every $1 \leq r_{2} \leq r_{1} \leq \infty,$ we have
\begin{equation*}\label{eq211}
\dot{\mathcal{N}}_{p_{2}, \mu_{2}, r_{2}}^{s_{2}} \hookrightarrow \dot{\mathcal{N}}_{p_{1}, \mu_{1}, r_{1}}^{s_{1}} \quad \text { and } \dot{\mathcal{N}}_{p_{2}, \mu_{2}, r_{2}}^{s_{2}} \hookrightarrow \dot{B}_{\infty, r_{2}}^{s_{2}-\frac{n-\mu_{2}}{p_{2}}}.
\end{equation*}
(iii)(H\"{o}lder inequality)\ \
Let $\frac{1}{p_{3}}=\frac{1}{p_{2}}+\frac{1}{p_{1}}$ and $\frac{\mu_{3}}{p_{3}}=\frac{\mu_{2}}{p_{2}}+\frac{\mu_{1}}{p_{1}}$.
If $f_{j} \in \mathcal{M}_{p_{j}, \mu_{j}}$ with $j=1,2,$ then $f_{1} f_{2} \in \mathcal{M}_{p_{3}, \mu_{3}}$ and
\begin{equation*}\label{eq212}
\left\|f_{1} f_{2}\right\|_{p_{3}, \mu_{3}} \leq\left\|f_{1}\right\|_{p_{1}, \mu_{1}}\left\|f_{2}\right\|_{p_{2}, \mu_{2}}.
\end{equation*}
\end{lem}


Set $\alpha=1$ in [\cite{ap}, Lemma 3.1],
we have the following decay estimates about
the heat semi-group in the Sobolev-Morrey
or Besov-Morrey space.
\begin{lem}\label{lem3}
Let $s, \beta \in \mathbb{R}, 1<p \leq q<\infty, 0 \leq \mu<n,$ and $(\beta-s)+\frac{n-\mu}{p}-\frac{n-\mu}{q}<2$ where $\beta \geq s.$\\
(i)\ \ There exists $C > 0$ such that
\begin{equation}
\left\|e^{t\Delta} f\right\|_{\mathcal{M}_{q, \mu}^{\beta}} \leq C t^{-\frac{1}{2}(\beta-s)-\frac{1}{2}\left(\frac{n-\mu}{p}-\frac{n-\mu}{q}\right)}\|f\|_{\mathcal{M}_{p, \mu}^{s}}\nonumber
\end{equation}
for every $t>0$ and $f \in \mathcal{M}_{p, \mu}^{s}$.\\
(ii)\ \ Let $r \in[1, \infty],$ there exists $C>0$ such that
\begin{equation}
\left\|e^{t\Delta} f\right\|_{\dot{\mathcal{N}}_{q, \mu, r}^{\beta}} \leq C t^{-\frac{1}{2}(\beta-s)-\frac{1}{2}\left(\frac{n-\mu}{p}-\frac{n-\mu}{q}\right)}\|f\|_{\dot{\mathcal{N}}_{p, \mu, r}^{s}}\nonumber
\end{equation}
for every $f \in \mathcal{S}^{\prime} / \mathcal{P}$ and $t>0$.\\
(iii)\ \ Let $r \in[1, \infty]$ and $\beta>s,$ there exists $C>0$ such that
\begin{equation}
\left\|e^{t\Delta} f\right\|_{\dot{\mathcal{N}}_{q, \mu, 1}^{\beta}} \leq C t^{-\frac{1}{2}(\beta-s)-\frac{1}{2}\left(\frac{n-\mu}{p}-\frac{n-\mu}{q}\right)}\|f\|_{\dot{\mathcal{N}}_{p, \mu, r}^{s}}\nonumber
\end{equation}
for every $f \in \mathcal{S}^{\prime} / \mathcal{P}$.
\end{lem}
\begin{rem}\label{remhw1}
Lemma \ref{lem3}
 also holds when $e^{t\Delta}$ is replaced by $e^{t A}$.
\end{rem}

\subsection{Proof of Theorem \ref{thm1}}

Before prove Theorem \ref{thm1}, we need  to prove the following proposition.
\begin{prop}\label{propmax201}
Let $1<p\leq q<\infty, 1<\alpha<\infty, 0\leq \mu< n$ and $s\in\mathbb{R}.$ Assume that
\begin{equation*}
\frac{n-\mu}{p}<\frac{2}{\alpha}+\frac{n-\mu}{q}.
\end{equation*}
For $a \in \dot{\mathcal{N}}_{p, \mu, r}^{k}$ with
 $k=s+2+[(n-\mu)/p-(n-\mu)/q]-2/\alpha,$ it holds that
\begin{equation*}
A e^{-t A} a \in L^{\alpha, r}(0, \infty ; \dot{\mathcal{N}}_{q,\mu, 1}^{s})
\end{equation*}
with the same estimate
\begin{equation*}
\left\| \| A e^{-t A} a\|_{\dot{\mathcal{N}}_{q, \mu, 1}^{s}}\right\|_{L^{\alpha, r}(0, \infty)} \leq C\|a\|_{\dot{\mathcal{N}}_{p, \mu,  r}^{k}}
\end{equation*}
where $C=C(n,\mu, q, \alpha, s, r)$.
\end{prop}
\begin{proof}
Since $\frac{n-\mu}{p}<\frac{2}{\alpha}+\frac{n-\mu}{q},$ we have that $k<s+2$.
Hence taking $\theta \in(0,1)$ and
$k_{0}<k<k_{1}<s+2$ so that $k=(1-\theta) k_{0}+\theta k_{1}.$
Using Lemma \ref{lem3}(iii), we have
\begin{equation*}
  \left\|A e^{-t A} a\right\|_{\dot{\mathcal{N}}_{q,\mu, 1}^{s}}
  =\left\|e^{-t A} a\right\|_{\dot{\mathcal{N}}_{q,\mu, 1}^{s+2}}
   \leq C t^{-\frac{1}{2}\left(\frac{n-\mu}{p}-\frac{n-\mu}{q}\right)-\frac{1}{2}\left(s+2-k_{i}\right)}
   \|a\|_{\dot{\mathcal{N}}_{p,\mu, \infty}^{k_{i}}}
\end{equation*}
for $i=0,1,$ and hence we see that the mapping
\begin{equation*}
  a \in \dot{\mathcal{N}}_{p,\mu, \infty}^{k_{i}} \mapsto
  \left\|A e^{-t A} a\right\|_{\dot{\mathcal{N}}_{q,\mu, 1}^{s}} \in L^{\alpha_{i}, \infty}(0, \infty)
\end{equation*}
is a bounded sub-additive operator for
\begin{equation*}
\frac{1}{\alpha_{i}}=\frac{1}{2}\left(\frac{n-\mu}{p}-\frac{n-\mu}{q}\right)
+\frac{1}{2}\left(s+2-k_{i}\right), \quad i=0,1.
\end{equation*}
Then it follows from the real interpolation theorem that
\begin{equation*}
a \in\left(\dot{\mathcal{N}}_{p,\mu, \infty}^{k_{0}}, \dot{\mathcal{N}}_{p,\mu, \infty }^{k_{1}}\right)_{\theta, r} \rightarrow\left\|A e^{-t A} a\right\|_{\dot{\mathcal{N}}_{q,\mu, 1}^{s}} \in\left(L^{\alpha_{0}, \infty}(0, \infty), L^{\alpha_{1}, \infty}(0, \infty)\right)_{\theta, r}.
\end{equation*}
Since $\left(\dot{\mathcal{N}}_{p,\mu, \infty}^{k_{0}}, \dot{\mathcal{N}}_{p, \mu, \infty}^{k_{1}}\right)_{\theta, r}=\dot{\mathcal{N}}_{p,\mu, r}^{k}$
(by (\ref{hweq1}))
and since
$\left(L^{\alpha_{0}, \infty}(0, \infty), L^{\alpha_{1}, \infty}(0, \infty)\right)_{\theta, r}=L^{\alpha, r}(0, \infty),$ implied by
\begin{eqnarray*}
\frac{1}{\alpha}
&=&\frac{1-\theta}{\alpha_{0}}+\frac{\theta}{\alpha_{1}}\nonumber\\
&=&(1-\theta)\left(\frac{1}{2}\left(\frac{n-\mu}{p}-\frac{n-\mu}{q}\right)
+\frac{1}{2}\left(s+2-k_{0}\right)\right)\nonumber\\
\quad \ \ \ \ \ &+&\theta\left(\frac{1}{2}\left(\frac{n-\mu}{p}-\frac{n-\mu}{q}\right)
+\frac{1}{2}\left(s+2-k_{1}\right)\right)\nonumber\\
&=&\frac{1}{2}\left(\frac{n-\mu}{p}-\frac{n-\mu}{q}\right)+\frac{1}{2}\left(s+2-(1-\theta) k_{0}-\theta k_{1}\right)\nonumber\\
&=&\frac{1}{2}\left(\frac{n-\mu}{p}-\frac{n-\mu}{q}\right)+\frac{1}{2}(s+2-k),
\end{eqnarray*}
we conclude that the mapping
\begin{equation*}
a \in \dot{\mathcal{N}}_{p,\mu, r}^{k} \rightarrow\left\|A e^{-t A} a\right\|_{\dot{\mathcal{N}}_{q,\mu, 1}^{s}} \in L^{\alpha, r}(0, \infty)
\end{equation*}
is a bounded sub-additive operator, which yields the desired result.
 This proves Proposition
\ref{propmax201}.
\end{proof}


\textit{\textbf{Proof of Theorem \ref{thm1}.}}

\textit{Step 1.}
Let us first prove in case $a=0.$
By the  maximal regularity
theorem in the homogeneous Sobolev-Morrey space(for the details, we can
refer to [\cite{a1}, Chapter 4] and [\cite{adams}, Chapter 8])
 $\mathcal{M}_{q, \mu}^{s}$ for $s_{0}<s<s_{1} \leq k+2$,
for every $f \in L^{\alpha}(0, T ; \mathcal{M}_{q, \mu}^{s})$
$(i=0,1)$ with $0<T \leq \infty$
there exists a unique solution $u$ of equation $(\ref{eqmaxs})$ in the class
\begin{equation*}
u_{t}, A u \in L^{\alpha}(0, T ; \mathcal{M}_{q, \mu}^{s_{i}})
\end{equation*}
with the estimate
\begin{equation*}
\left\|u_{t}\right\|_{L^{\alpha}(0, T ; \mathcal{M}_{q, \mu}^{s_{i}})}+\|A u\|_{L^{\alpha}(0, T ; \mathcal{M}_{q, \mu}^{s_{i}})}
\leq C\|f\|_{L^{\alpha}(0, T ; \mathcal{M}_{q, \mu}^{s_{i}})},
\ \ i= 0, 1,
\end{equation*}
where $C=C(n,\mu, q, p, \alpha, s_{0}, s_{1})$ is independent of $T$.
This implies that the mapping
\begin{equation}
S : f \in L^{\alpha}(0, T ; \mathcal{M}_{q, \mu}^{s_{i}}) \rightarrow\left(u_{t}, A u\right) \in L^{\alpha}(0, T ; \mathcal{M}_{q, \mu}^{s_{i}})^{2},\ \ i=0, 1\nonumber
\end{equation}
is a bounded linear operator with its operator norm independent of $T$.
Hence by the real interpolation, $S$ extends a bounded operator from
$L^{\alpha}(0, T ;(\mathcal{M}_{q, \mu}^{s_{0}}, \mathcal{M}_{q, \mu}^{s_{1}})_{\theta, \beta})$
to $L^{\alpha}(0, T ;(\mathcal{M}_{q, \mu}^{s_{0}}, \mathcal{M}_{q, \mu}^{s_{1}})_{\theta, \beta})^{2}$ for all $1 \leq \beta \leq \infty.$
Since $(\mathcal{M}_{q, \mu}^{s_{0}}, \mathcal{M}_{q, \mu}^{s_{1}})_{\theta, \beta}=\dot{\mathcal{N}}_{q, \mu, \beta}^{s}$
(by (\ref{eq208}))
 with $s=(1-\theta) s_{0}+\theta s_{1},$ we see that
\begin{equation*}
S : f \in L^{\alpha}(0, T ; \dot{\mathcal{N}}_{q, \mu, \beta}^{s}) \rightarrow\left(u_{t}, A u\right) \in L^{\alpha}(0, T ; \dot{\mathcal{N}}_{q, \mu, \beta}^{s})^{2}
\end{equation*}
is a bounded operator with its operator norm independent of $T$.
Taking $\alpha_{0}<\alpha<\alpha_{1}$ and
$0<\theta<1$ so that $1 / \alpha=(1-\theta) / \alpha_{0}+\theta / \alpha_{1},$ we see that
\begin{equation*}
\begin{aligned} S : f & \in\left(L^{\alpha_{0}}(0, T ; \dot{\mathcal{N}}_{q, \mu, \beta}^{s}), L^{\alpha_{1}}(0, T ; \dot{\mathcal{N}}_{q, \mu, \beta}^{s})\right)_{\theta, r} \\ & \rightarrow\left(u_{t}, A u\right) \in\left(L^{\alpha_{0}}(0, T ; \dot{\mathcal{N}}_{q, \mu, \beta} ^{s}), L^{\alpha_{1}}(0, T ; \dot{\mathcal{N}}_{q, \mu, \beta}^{s})\right)_{\theta, r}^{2} \end{aligned}
\end{equation*}
is a bounded operator with its operator norm independent of $T$. Since
\begin{equation}
\left(L^{\alpha_{0}}(0, T ; \dot{\mathcal{N}}_{q, \mu, \beta}^{s}), L^{\alpha_{1}}(0, T ; \dot{\mathcal{N}}_{q, \mu, \beta}^{s})\right)_{\theta, r}=L^{\alpha, r}(0, T ; \dot{\mathcal{N}}_{q, \mu, \beta}^{s})\nonumber
\end{equation}
we obtain the desired result with the estimate (\ref{eqmax103}) for $a = 0$.

Step $2 .$ For $a \in \dot{\mathcal{N}}_{p, \mu, r}^{k}$
and $f \in L^{\alpha, r}(0, T ; \dot{\mathcal{N}}_{q,\mu, \beta}^{s}),$ we see that
\begin{equation*}
u(t)=e^{-t A} a+S f(t), \quad 0<t<T
\end{equation*}
solves  equation $(\ref{eqmaxs})$.
Since $\dot{\mathcal{N}}_{q,\mu, 1}^{s} \subset \dot{\mathcal{N}}_{q,\mu, \beta}^{s}$,
the desired result with the estimate (\ref{eqmax103}) is a consequence of
Proposition \ref{propmax201} and the argument of the above Step 1.
This completes the proof of Theorem \ref{thm1}.

\section{Some estimates
  for the nonlinear term }
In this section,
we establish several bilinear
estimates associated with the nonlinear term $u\cdot\nabla u$
in terms with the norms of $u_{t}$ and $Au$ in $L^{\alpha, r}(0, T ; \dot{\mathcal{N}}_{q,\mu, \infty}^{s})$.

We first recall the following variant of the Hardy-Littlewood-Sobolev inequality.
(see \cite{rs}, Chapter IX.4 and \cite{ks6}).
\begin{prop}\label{prop31}
For $0<\sigma<1,$ we define $I_{\sigma} g$ by
\begin{equation}
I_{\sigma} g(t)=\int_{0}^{\infty}|t-\tau|^{\sigma-1} g(\tau) d \tau, \quad 0<t<\infty.\nonumber
\end{equation}
For $g \in L^{\alpha, q}(0, \infty)$ with $1<\alpha<1 / \sigma, 1 \leq q \leq \infty$,
it holds that $I_{\sigma} g \in L^{\alpha_{*}, q}(0, \infty)$ for
$\alpha<\alpha_{*}<\infty$ satisfying $1 / \alpha_{*}=1 / \alpha-\sigma$ with the estimate
\begin{equation}
\left\|I_{\sigma} g\right\|_{L^{\alpha_{*}, q}(0, \infty)} \leq C\|g\|_{L^{\alpha, q}(0, \infty)}\nonumber
\end{equation}
where $C=C(\alpha, \sigma)$.
\end{prop}

Next, we present   the Leibniz rule of the fractional derivatives in
the homogeneous Besov-Morrey spaces.

\begin{prop}\label{prop32}

 Let $1 \leq p \leq \infty, 1 \leq q \leq \infty,$ and let $s>0, \sigma_{1}>0,  \sigma_{2}>0.$ We take $1 \leq p_{1}, p_{2} \leq \infty$
and $1 \leq r_{1}, r_{2} \leq \infty $ so that $1 / p=1 / p_{1}+1 / p_{2}=1 / r_{1}+1 / r_{2} $
 and $\mu/ p=\mu_{1} / p_{1}+\mu_{2} / p_{2}=\nu_{1} / r_{1}+\nu_{2}/ r_{2}.$ For every $f \in \dot{\mathcal{N}}_{p_{1},\mu_{1}, q}^{s+\sigma_{1}} \cap \dot{\mathcal{N}}_{r_{1},\nu_{1}, \infty}^{-\sigma_{2}}$ and
$g \in \dot{\mathcal{N}}_{p_{2}, \mu_{2}, \infty}^{-\sigma_{1}} \cap \dot{\mathcal{N}}_{r_{2},\nu_{2}, q}^{s+\sigma_{2}},$ it holds that $f \cdot g \in \dot{\mathcal{N}}_{p,\mu, q}^{s}$ with the estimate
\begin{equation}
\|f \cdot g\|_{\dot{\mathcal{N}}_{p,\mu, q}^{s}}\leq C\left(\|f\|_{\dot{\mathcal{N}}_{p_{1}, \mu_{1}, q}^{s+\sigma_{1}}}\|g\|_{\dot{\mathcal{N}}_{p_{2},\mu_{2}, \infty}^{-\sigma_{1}}}
+\|f\|_{\dot{\mathcal{N}}_{r_{1},\nu_{1}, \infty}^{-\sigma_{2}}}\|g\|_{\dot{\mathcal{N}}_{r_{2},\nu_{2}, q}^{s+\sigma_{2}}}\right).\nonumber
\end{equation}
where $C=C\left(n, p, p_{1}, p_{2}, r_{1}, r_{2}, s, \sigma_{1}, \sigma_{2}, \mu_{1}, \mu_{2}, \nu_{1}, \nu_{2}\right).$
\end{prop}

\begin{proof}
The proof of Proposition \ref{prop32}
is essentially adapted from [\cite{ks1}, Lemma 2.1],
for completeness, we present the proof here.

Using Bony's paraproduct formula,
we have
\begin{eqnarray}
  fg=\sum_{k=-\infty}^{+\infty}\Delta_{k}f S_{k}g
  +\sum_{k=-\infty}^{+\infty}\Delta_{k}g S_{k}f
  +\sum_{k=-\infty}^{+\infty}\sum_{|l-k|\leq2}\Delta_{k}f \Delta_{l}g
  =h_{1}+h_{2}+h_{3},\nonumber
\end{eqnarray}
where $\Delta_{k}f=\varphi_{k}*f$
and $S_{k}f=\Sigma_{l=-\infty}^{k-3}\Delta_{l}f.$

Using the support properties of $\varphi_{k}(\xi)$,
the Minkowski inequality and Lemma \ref{lem1}(iii), one obtains
\begin{eqnarray}
\|h_{1}\|_{\dot{\mathcal{N}}_{p,\mu, q}^{s}}
&=&\|\{2^{js}\|\Delta_{j}(\sum_{k=-\infty}^{+\infty}\Delta_{k}f S_{k}g)\|_{p,\mu}\}_{j}\|_{l^{q}}\nonumber\\
&=&\|\{2^{js}\|\Delta_{j}(\sum_{k=j-3}^{j+3}\Delta_{k}f S_{k}g)\|_{p,\mu}\}_{j}\|_{l^{q}}\nonumber\\
&=&\|\{2^{js}\|\Delta_{j}(\sum_{l=-3}^{3}\Delta_{j+l}f S_{j+l}g)\|_{p,\mu}\}_{j}\|_{l^{q}}\nonumber\\
&\leq&
\sum_{l=-3}^{3}\|\{2^{js}\|\Delta_{j}(\Delta_{j+l}f S_{j+l}g)\|_{p,\mu}\}_{j}\|_{l^{q}}\nonumber\\
&\leq&
\sum_{l=-3}^{3}\|\{2^{js}\|M(\Delta_{j+l}f S_{j+l}g)\|_{p,\mu}\}_{j}\|_{l^{q}}\nonumber\\
&\leq&
\sum_{l=-3}^{3}\|\{2^{js}\|\Delta_{j+l}f S_{j+l}g\|_{p,\mu}\}_{j}\|_{l^{q}}\nonumber
\end{eqnarray}
\begin{eqnarray}
&=&\sum_{l=-3}^{3}\|\{2^{(i-l)s}\|\Delta_{i}f S_{i}g\|_{p,\mu}\}_{i}\|_{l^{q}}\nonumber\\
&\leq&
C\|\{2^{is}\|\Delta_{i}f \sum_{j=-\infty}^{i-3}\Delta_{j}g\|_{p,\mu}\}_{i}\|_{l^{q}}\nonumber\\
&=&C\|\{2^{i(s+\alpha)}\|\Delta_{i}f \sum_{j=-\infty}^{i-3}2^{-\alpha j}\Delta_{j}g \|_{p,\mu}2^{\alpha(j-i)}\}_{i}\|_{l^{q}}\nonumber\\
&\leq&
C\|\{2^{i(s+\alpha)}\|\Delta_{i}f\|_{p_{1}, \mu_{1}} \sum_{j=-\infty}^{i-3}2^{-\alpha j}\|\Delta_{j}g \|_{p_{2}, \mu_{2}}2^{\alpha(j-i)}\}_{i}\|_{l^{q}}\nonumber\\
&\leq&
C\|\{2^{i(s+\alpha)}\|\Delta_{i}f\|_{p_{1}, \mu_{1}} (\sup_{j}2^{-\alpha j}\|\Delta_{j}g \|_{p_{2}, \mu_{2}})\sum_{j\leq i-3}2^{\alpha(j-i)}\}_{i}\|_{l^{q}}\nonumber\\
&\leq& C\|f\|_{\dot{\mathcal{N}}_{p_{1}, \mu_{1}, q}^{s+\sigma_{1}}}\|g\|_{\dot{\mathcal{N}}_{p_{2},\mu_{2}, \infty}^{-\sigma_{1}}},\nonumber
\end{eqnarray}
where  $M$ in the second inequality denotes the Hardy-Littlewood
maximal operator and is bounded
in the Morrey space when $1<p<\infty$ and $0\leq\mu<n,$
we can refer to [\cite{adams}, Chapter 6] for the details.

From these estimates, we obtain
\begin{eqnarray}
\left\|h_{1}\right\|_{\dot{\mathcal{N}}{p,\mu, q}^{s}} \leq C\|f\|_{\dot{\mathcal{N}}_{p_{1}, \mu_{1}, q}^{s+\alpha}}\|g\|_{\dot{\mathcal{N}}_{p_{2},\mu_{2}, \infty}^{-\alpha}},\nonumber
\end{eqnarray}
where $ 1<p<\infty$,
$1<q<\infty$ $ 1<p_{1}<\infty, 1<p_{2} \leq \infty$  with
$ 1 / p=1 / p_{1}+1 / p_{2}$
and
 $\mu / p=\mu_{1} / p_{1}+\mu_{2} / p_{2}. $

If we replace the role of $f$ and $g$ in $h_{1}$ with that of $f$ and $g$ in $h_{2}$ ,
 respectively, we obtain
  the following similar  estimate as above:
\begin{eqnarray}
\left\|h_{2}\right\|_{\dot{\mathcal{N}}{p,\mu, q}^{s}} \leq C\|f\|_{\dot{\mathcal{N}}_{r_{1}, \nu_{1}, q}^{-\beta}}\|g\|_{\dot{\mathcal{N}}_{r_{2},\nu_{2}, \infty}^{s+\beta}},\nonumber
 \end{eqnarray}
    where  $1<p<\infty$,
   $1<q<\infty$ $1<r_{1}<\infty, 1<r_{2} \leq \infty$
    with $ 1 / p=1 / r_{1}+1 / r_{2}$
   and
   $\mu / p=\nu_{1} / r_{1}+\nu_{2} / r_{2}. $

 Since
 \begin{equation*}
 \operatorname{supp} \mathcal{F}((\varphi_{k} * f)(\varphi_{j} * g)) ) \subset\{\xi \in \mathbb{R}^{n} ;|\xi| \leq 2^{\max (k, l)+2}\},
 \end{equation*}
 there holds $\varphi_{j} *(\varphi_{k} * f)(\varphi_{j} * g)=0$ for $\max (k, l)+2 \leq j-1$.
 Thus, we can deal with $h_{3}$ as
\begin{eqnarray}
\left\|h_{3}\right\|_{\dot{\mathcal{N}}{p,\mu, q}^{s}}
&\leq&
\|\{2^{sj}\|\sum_{\operatorname{max}(k, l)+2 \geq j} \sum_{|k-l| \leq 2} \Delta_{j} \left(\Delta_{k}f\right)\left(\Delta_{l}g\right)\|_{p,\mu}\}_{j}\|_{l^{q}}\nonumber\\
&=&\|\{2^{sj}\|\sum_{r\geq -4} \sum_{|t| \leq 2} \Delta_{j} \left(\Delta_{j+r}f\right)\left(\Delta_{j+r+t}g\right)\|_{p,\mu}\}_{j}\|_{l^{q}}\nonumber\\
&\leq&
\sum_{r\geq -4} \sum_{|t| \leq 2}\|\{2^{sj}\| \Delta_{j} \left(\Delta_{j+r}f\right)\left(\Delta_{j+r+t}g\right)\|_{p,\mu}\}_{j}\|_{l^{q}}\nonumber\\
&\leq&
\sum_{r\geq -4} \sum_{|t| \leq 2}\|\{2^{sj}\| M
\left(\Delta_{j+r}f\right)\left(\Delta_{j+r+t}g\right)\|_{p,\mu}\}_{j}\|_{l^{q}}\nonumber
\end{eqnarray}
\begin{eqnarray}
&\leq&
\sum_{r\geq -4} \sum_{|t| \leq 2}\|\{2^{sj}\|
\left(\Delta_{j+r}f\right)\left(\Delta_{j+r+t}g\right)\|_{p,\mu}\}_{j}\|_{l^{q}}\nonumber\\
&\leq&
\sum_{r\geq -4} \sum_{|t| \leq 2}2^{-s r} 2^{\alpha t}\|\{\|
2^{(s+\alpha)(j+r)}\left(\Delta_{j+r}f\right)2^{-\alpha(j+r+t)}
\left(\Delta_{j+r+t}g\right)\|_{p,\mu}\}_{j}\|_{l^{q}}\nonumber\\
&\leq&
     \|\{\|2^{(s+\alpha)(i)}\left(\Delta_{i}f\right)\|_{p_{1},\mu_{1}}\}_{i}\|_{l^{q}}
       \|\{2^{-\alpha i}\|\left(\Delta_{i}f\right)\|_{p_{2},\mu_{2}}\}_{i}\|_{l^{\infty}}
       \sum_{r \geq-4} 2^{-s r}\nonumber\\
&\leq &C\|f\|_{\dot{\mathcal{N}}_{p_{1},\mu_{1}, q}^{s+\alpha}}\|g\|_{\dot{F}_{p_{2}, \mu_{2},\infty}^{-\alpha}}.\nonumber
\end{eqnarray}
   This proves Proposition \ref{prop32}.
    \end{proof}

The following lemma may be regarded as an embedding theorem in terms of the graph norm
associated with the maximal Lorentz regularity in Besov-Morrey spaces. Corresponding estimate to
that of usual Lebesgue spaces was proved by Giga-Sohr \cite{gs}
and \cite{ks5}
and to  that of Besov space was proved by  \cite{ks6}.

\begin{lem}\label{lemmax31}
Let $1<q \leq q_{0}<\infty, 0\leq\mu<n, 1<\alpha<\alpha_{0}<\infty,-\infty<s<s_{0}<\infty$ be as
\begin{equation*}
\frac{2}{\alpha_{0}}+\frac{n-\mu}{q_{0}}-s_{0}=\frac{2}{\alpha}+\frac{n-\mu}{q}-s-2.
\end{equation*}
Let $1<p<\infty$ satisfy
\begin{equation*}
  \frac{n-\mu}{q} \leq \frac{n-\mu}{p}<\frac{2}{\alpha_{0}}+\frac{n-\mu}{q_{0}}
  \left(=\frac{2}{\alpha}+\frac{n-\mu}{q}-s-2+s_{0}\right).
\end{equation*}
Suppose that a measurable function $u$ on $\mathbb{R}^{n} \times(0, T)$ satisfies that
\begin{equation}
u_{t},  Au \in L^{\alpha, r}(0, T ; \dot{\mathcal{N}}_{q,\mu, \infty}^{s})\nonumber
\end{equation}
with $u(0)=a \in \dot{\mathcal{N}}_{p,\mu, r}^{k}$ for $k=2+(n-\mu)/p-(2/\alpha+(n-\mu)/q-s) .$ Then it holds that
\begin{equation}
u \in L^{\alpha_{0}, r}(0, T ; \dot{\mathcal{N}}_{q_{0},\mu, 1}^{s_{0}})\nonumber
\end{equation}
with the estimate
\begin{eqnarray}
&&\|u\|_{L^{\alpha_{0}, r}(0, T ; \dot{\mathcal{N}}_{q_{0},\mu, 1}^{s_{0}})}\nonumber\\
&&\leq\left\|e^{-t A} a\right\|_{L^{\alpha_{0}, r}(0, T ; \dot{\mathcal{N}}_{q_{0},\mu, 1}^{s_{0}})}
+C\left(\left\|u_{t}\right\|_{L^{\alpha, r}(0, T ; \dot{\mathcal{N}}_{q,\mu, \infty}^{s})}+\|A u\|_{L^{\alpha, r}(0, T ; \dot{\mathcal{N}}_{q,\mu, \infty}^{s})}\right),\label{eqmax301}\\
&&\left\|e^{-t A} a\right\|_{L^{\alpha_{0}, r}(0, \infty ; \dot{\mathcal{N}}_{q_{0},\mu, 1}^{s_{0}})}
 \leq C\|a\|_{\dot{\mathcal{N}}_{p,\mu, r}^{k}},\label{eqmax302}
\end{eqnarray}
where $C=C(n,\mu, q, \alpha, s, p, r)$ is independent of $u, a$ and $T$.
\end{lem}
\begin{proof}
Let $f(t)=u_{t}+A u$.
Then it holds that $u(t)=U(t) a+F f(t)$, where $U(t) a=e^{-t A} a$
and $F f(t)=\int_{0}^{t} e^{-(t-\tau) A} f(\tau) d \tau$.
By the assumption, we have $f \in L^{\alpha, r}(0, T ; \dot{\mathcal{N}}_{q,\mu, \infty}^{s})$ with the
estimate
\begin{equation}\label{eqmax303}
\|f\|_{L^{\alpha, r}(0, T ; \dot{\mathcal{N}}_{q,\mu, \infty}^{s})} \leq\left\|u_{t}\right\|_{L^{\alpha, r}(0, T ; \dot{\mathcal{N}}_{q,\mu, \infty}^{s})}+\|A u\|_{L^{\alpha, r}(0, T ;\dot{\mathcal{N}}_{q,\mu, \infty}^{s})}.
\end{equation}
Since $q \leq q_{0}$ and $s<s_{0},$ we have by Lemma \ref{lem3}(iii) that
\begin{eqnarray*}
\|F f(t)\|_{\dot{\mathcal{N}}_{q_{0}, \mu, 1}^{s_{0}}}
&\leq& \int_{0}^{t}\left\|e^{-(t-\tau) A} f(\tau)\right\|_{\dot{\mathcal{N}}_{q_{0},\mu, 1}^{s_{0}}}
 \mbox{d}\tau\\
&\leq& C \int_{0}^{t}(t-\tau)^{-\frac{1}{2}(\frac{n-\mu}{q}-\frac{n-\mu}{q_{0}})-\frac{1}{2}
(s_{0}-s)}\|f(\tau)\|_{\dot{\mathcal{N}}_{q,\mu, \infty}^{s}}\mbox{d}\tau\\
&\leq& C \int_{0}^{t}(t-\tau)^{\sigma-1}\|f(\tau)\|_{\dot{\mathcal{N}}_{q,\mu, \infty}^{s}} \mbox{d}\tau,
\end{eqnarray*}
where $\sigma \equiv 1-\frac{1}{2}(\frac{n-\mu}{q}-\frac{n-\mu}{q_{0}})-\frac{1}{2}(s_{0}-s)$.
Since $2 / \alpha_{0}+(n-\mu)/q_{0}-s_{0}=2 / \alpha+(n-\mu)/q-s-2$ and since
$1<\alpha<\alpha_{0},$ we have
\begin{eqnarray}
0<\frac{1}{2}\left(\frac{n-\mu}{q}-\frac{n-\mu}{q_{0}}\right)+\frac{1}{2}\left(s_{0}-s\right)
=\frac{1}{2}\left\{\left(\frac{n-\mu}{q}-s\right)-\left(\frac{n-\mu}{q_{0}}-s_{0}\right)\right\}\nonumber\\
=\frac{1}{2}\left(\frac{2}{\alpha_{0}}-\frac{2}{\alpha}+2\right)
=1+\left(\frac{1}{\alpha_{0}}-\frac{1}{\alpha}\right)<1,\nonumber
\end{eqnarray}
\begin{equation}
\sigma
=1-\frac{1}{2}\left(\frac{n-\mu}{q}-\frac{n-\mu}{q_{0}}\right)
-\frac{1}{2}\left(s_{0}-s\right)
=\frac{1}{\alpha}
-\frac{1}{\alpha_{0}}<\frac{1}{\alpha},\nonumber
\end{equation}
which yields $0<\sigma<1$ and $1<\alpha<1 / \sigma$.
Since $1 / \alpha_{0}=1 / \alpha-\sigma,$ we may apply Proposition
\ref{prop31} with $\alpha_{*}=\alpha_{0}$ and with $g(\tau)=\|f(\tau)\|_{\dot{\mathcal{N}}_{q,\mu, \infty}^{s}}$
for $0<\tau \leq T, g(\tau)=0$
for $T \leq \tau$ so that
$g \in L^{\alpha, r}(0, \infty)$
and obtain that
$\|F f(\cdot)\|_{\dot{\mathcal{N}}_{q_{0},\mu, 1}^{s_{0}}} \in L^{\alpha_{0}, r}(0, T)$
with the estimate
\begin{equation*}
\left\| \left\| F f(\cdot)\right\|_{\dot{\mathcal{N}}_{q_{0},\mu, 1}^{s_{0}}}\right\|_{L^{\alpha_{0}, r}(0, T)}
 \leq C\left\| \left\| f(\cdot)\right\|_{\dot{\mathcal{N}}_{q,\mu, \infty}^{s}}\right\|_{L^{\alpha, r}(0, T)}
\end{equation*}
which yields with the aid of (\ref{eqmax303}) that
\begin{equation}\label{eqmax304}
\|F f\|_{L^{\alpha_{0}, r}(0, T ; \dot{\mathcal{N}}_{q_{0},\mu, 1}^{s_{0}})}
\leq
C\left(\left\|u_{t}\right\|_{L^{\alpha, r}(0, T ; \dot{\mathcal{N}}_{q,\mu, \infty}^{s})}+\|A u\|_{L^{\alpha, r}(0, T ; \dot{\mathcal{N}}_{q,\mu, \infty}^{s})}\right),
\end{equation}
where $C=C\left(n, q, \alpha, s, q_{0}, \alpha_{0}, s_{0}, r\right)$ is independent of $T$.
We next show that $U(\cdot) a \in L^{\alpha_{0}, r}\left(0, \infty ; \dot{\mathcal{N}}_{q_{0}, \mu, 1}^{s_{0}}\right)$ for
$a \in \dot{\mathcal{N}}_{p,\mu, r}^{k}$. By the assumption, we have that
\begin{equation*}
p \leq q \leq q_{0}, \quad k=(n-\mu) / p-(2 / \alpha+(n-\mu) / q-s-2)<s_{0},
\end{equation*}
and hence we may take $k_{1}, k_{2}$ and $0<\theta<1$ so that $k_{1}<k<k_{2}<s_{0}$
and so that $k=(1-\theta) k_{1}+\theta k_{2}.$ For $a \in \dot{\mathcal{N}}_{p,\mu, \infty}^{k_{i}}, i=1,2$, it holds by Lemma \ref{lem3}(iii) that
\begin{equation}
\left\|e^{-t A} a\right\|_{\dot{\mathcal{N}}_{q_{0},\mu, 1}^{s_{0}}} \leq C t^{-\frac{1}{2}(\frac{n-\mu}{p}-\frac{n-\mu}{q_{0}})
-\frac{1}{2}(s_{0}-k_{i})}\|a\|_{\dot{\mathcal{N}}_{p,\mu, \infty}^{k_{i}}},\ \ i=1, 2\nonumber
\end{equation}
for all $t>0 .$ Hence the map $S_{t}$
\begin{equation}
S_{t} : a \in \dot{\mathcal{N}}_{p,\mu, \infty}^{k_{i}} \rightarrow\left\|e^{-t A} a\right\|_{\dot{\mathcal{N}}_{q_{0},\mu, 1}^{s_{0}}} \in L^{\alpha_{i}, \infty}(0, \infty)\nonumber
\end{equation}
is a bounded sub-additive operator for
$1 / \alpha_{i}=\frac{1}{2}\left((n-\mu) / p-(n-\mu) / q_{0}\right)+\frac{1}{2}\left(s_{0}-k_{i}\right), i=1,2.$
By the real interpolation, $S_{t}$ extends to the map
\begin{equation*}
S_{t} : a \in\left(\dot{\mathcal{N}}_{p,\mu, \infty}^{k_{1}}, \dot{\mathcal{N}}_{p,\mu, \infty}^{k_{2}}\right)_{\theta, r} \rightarrow\left\|e^{-t A} a\right\|_{\dot{\mathcal{N}}_{q_{0},\mu, 1}^{s_{0}}} \in\left(L^{\alpha_{1}, \infty}(0, \infty), L^{\alpha_{2}, \infty}(0, \infty)\right)_{\theta, r}.
\end{equation*}
Since $\left(\dot{\mathcal{N}}_{p,\mu, \infty}^{k_{1}}, \dot{\mathcal{N}}_{p, \mu, \infty}^{k_{2}}\right)_{\theta, r}=\dot{\mathcal{N}}_{p,\mu, r}^{k}$, and since
$\left(L^{\alpha_{1}, \infty}(0, \infty), L^{\alpha_{2}, \infty}(0, \infty)\right)_{\theta, r}=L^{\alpha_{0}, r}(0, \infty)$ implied by
$(1-\theta) / \alpha_{1}+\theta / \alpha_{2}=\frac{1}{2}\left((n-\mu) / p-(n-\mu) / q_{0}\right)+\frac{1}{2}\left(s_{0}-k\right)=1 / \alpha_{0}$, we obtain the estimate
\begin{equation}\label{eqmax305}
\left\| \left\| e^{-t A} a\right\|_{\dot{\mathcal{N}}_{q_{0},\mu, 1}^{s_{0}}}\right\|_{L^{\alpha_{0}, r}(0, \infty)} \leq C\|a\|_{\dot{\mathcal{N}}_{p,\mu, r}^{k}}
\end{equation}
where $C=C\left(n, q, \alpha, s, q_{0}, \alpha_{0}, s_{0}, r\right)$.
Now the desired estimates (\ref{eqmax301}) and
(\ref{eqmax302}) are consequences of (\ref{eqmax304}) and (\ref{eqmax305}).
This proves Lemma \ref{lemmax31}.
\end{proof}

The following lemma is a bilinear estimate for the nonlinear term $u \cdot \nabla v$
in (\ref{ns}) in terms of the graph norm associated with the maximal Lorentz regularity theorem of the Stokes operator $A$. The space $L^{\alpha, r}(0, \infty ; \dot{\mathcal{N}}_{q,\mu, \infty}^{s})$
for $2 / \alpha+(n-\mu) / q-s=3$ is closely related to the scaling invariance with respect to equation (\ref{ns}).

\begin{lem}\label{lemmax32}
Let $1<q<\infty, 1<\alpha<\infty, 0\leq\mu<n$  and $ -1<s<\infty$ satisfy $2 / \alpha+(n-\mu) / q-s=3$.
Let $1<p \leq q$ and $1 \leq r \leq \infty.$
For measurable functions $u$ and $v$ in $\mathbb{R}^{n} \times(0, T), 0<T \leq \infty$ with
the properties that
\begin{eqnarray}
u_{t}, Au, v_{t},  Av \in L^{\alpha, r}(0, T ; \dot{\mathcal{N}}_{q, \mu, \infty}^{s}),\\
u(0)=a, v(0)=b \in \dot{\mathcal{N}}_{p,\mu, r}^{-1+(n-\mu) / p}
\end{eqnarray}
it holds that $ \mathbb{P}(u \cdot \nabla v) \in L^{\alpha, r}(0, T ; \dot{\mathcal{N}}_{q, \mu, \infty}^{s})$ with the estimate
\begin{eqnarray}\label{eqmax306}
&&\|\mathbb{P}(u \cdot \nabla v)\|_{L^{\alpha, r}(0, T ; \dot{\mathcal{N}}_{q,\mu, \infty}^{s})}\nonumber\\
&&\leq C\left(\left\|e^{-t A} a\right\|_{L^{\alpha_{0}, r}(0, T ; \dot{\mathcal{N}}_{q_{0},\mu, 1}^{s_{0}})}+\left\|u_{t}\right\|_{L^{\alpha, r}(0, T ; \dot{\mathcal{N}}_{q, \mu, \infty}^{s})}+\|A u\|_{L^{\alpha, r}(0, T ; \dot{\mathcal{N}}_{q,\mu, \infty}^{s})}\right)\nonumber\\
&&\times\left(\left\|e^{-t A} b\right\|_{L^{\alpha_{0}, r}(0, T ; \dot{\mathcal{N}}_{q_{0}, \mu, 1}^{s_{0}})}+\left\|v_{t}\right\|_{L^{\alpha, r}(0, T ; \dot{\mathcal{N}}_{q,\mu, \infty}^{s})}+\|A v\|_{L^{\alpha, r}(0, T ; \dot{\mathcal{N}}_{q,\mu,  \infty}^{s})}\right)
\end{eqnarray}
for some $\alpha<\alpha_{0}<\infty, q \leq q_{0}<\infty$ and $s<s_{0}<\infty$
such that $2 / \alpha_{0}+(n-\mu) / q_{0}-s_{0}=1$ where
$C=C\left(n,\mu, q, \alpha, q_{0}, \alpha_{0}, p, r\right)$
is a constant independent of $0<T \leq \infty$.
\end{lem}

\begin{proof}
Let us take $\alpha_{0}>\alpha, q_{0} \geq q$ and $s_{0} \in \mathbb{R}$ so that
\begin{equation}\label{eqmax307}
\alpha_{0}=2 \alpha, \quad \max \{(n-\mu) / p-1 / \alpha, s+2-1 / \alpha\}<(n-\mu)/ q_{0}, \quad \max \{s+1, (n-\mu) / p-1\}<s_{0}
\end{equation}
satisfying
\begin{equation}\label{eqmax308}
2 / \alpha_{0}+(n-\mu) / q_{0}-s_{0}=1.
\end{equation}
Since $-1<s,$ we have by $(\ref{eqmax307})$ that $0<s+1<s_{0}$,
and hence there exists $\sigma \in \mathbb{R}$ such that
\begin{equation}\label{eqmax309}
0<\sigma<s_{0}-(s+1).
\end{equation}
Let us define $q_{1}$ and $q_{2}$ by
\begin{eqnarray}\label{eqmax310}
s_{0}-\frac{n-\mu}{q_{0}}=s+1+\sigma-\frac{n-\mu}{q_{1}},\\
s_{0}-\frac{n-\mu}{q_{0}}=-\sigma-\frac{n-\mu}{q_{2}}.
\end{eqnarray}
By $(\ref{eqmax309})$ we have that $q_{0}<q_{1}$ and $q_{0}<q_{2}$.
Hence it follows that from Lemma \ref{lem3}(iii) that
\begin{equation}\label{eqmax311}
\dot{\mathcal{N}}_{q_{0},\mu, \infty}^{s_{0}} \hookrightarrow \dot{\mathcal{N}}_{q_{1},\mu, \infty}^{s+1+\sigma},
 \quad \dot{\mathcal{N}}_{q_{0}, \mu, \infty}^{s_{0}} \hookrightarrow \dot{\mathcal{N}}_{q_{2},\mu, \infty}^{-\sigma}.
\end{equation}
Furthermore, since $2 / \alpha_{0}+(n-\mu) / q_{0}-s_{0}=1$ and $\alpha_{0}=2 \alpha,$ we have that
\begin{equation}\label{eqmax312}
1 / q_{1}+1 / q_{2}=1 / q.
\end{equation}

Then it follows from (\ref{eqmax311}), (\ref{eqmax312}) and Proposition \ref{prop32} that
\begin{eqnarray*}
&&\|\mathbb{P}(u \cdot \nabla v)\|_{\dot{\mathcal{N}}_{q,\mu, \infty}^{s}}
=\|\nabla \cdot \mathbb{P}(u \otimes v)\|_{\dot{\mathcal{N}}_{q,\mu, \infty}^{s}}
=\|\mathbb{P}(u \otimes v)\|_{\dot{\mathcal{N}}_{q,\mu, \infty}^{s+1}}
\leq C\|(u \otimes v)\|_{\dot{\mathcal{N}}_{q,\mu, \infty}^{s+1}}\\
&&\leq C(\|u\|_{\dot{\mathcal{N}}_{q_{1},\mu, \infty}^{s+1+\sigma}}\|v\|_{\dot{\mathcal{N}}_{q_{2},\mu, \infty}^{-\sigma}}
+\|u\|_{\dot{\mathcal{N}}_{q_{2},\mu, \infty}^{-\sigma}}\|v\|_{\dot{\mathcal{N}}_{q_{1},\mu, \infty}^{s+1+\sigma}})\\
&&\leq C\|u\|_{\dot{\mathcal{N}}_{q_{0},\mu, \infty}^{s_{0}}}\|v\|_{\dot{\mathcal{N}}_{q_{0}, \mu,\infty}^{s_{0}}}.
\end{eqnarray*}
Since $\alpha_{0}=2 \alpha$ and since $ L^{\alpha, r}(0, T) \hookrightarrow L^{\alpha, 2r}(0, T)$, we have by the H\"{o}lder inequality in the Lorentz space on $(0, T)$ that
\begin{eqnarray*}
&&\left\| \| \mathbb{P}(u, \cdot \nabla v)\|_{\dot{\mathcal{N}}_{q, \mu, \infty}^{s}}\right\|_{L^{\alpha, r}(0, T)}\\
&&\leq C
\left\|\|u\|_{\dot{\mathcal{N}}_{q_{0},\mu, \infty}^{s_{0}}}\right\|_{L^{\alpha_{0}, 2r}(0, T)}\left\|\|v\|_{\dot{\mathcal{N}}_{q_{0},\mu, \infty}^{s_{0}}}\right\|_{L^{\alpha_{0}, 2r}(0, T)}\\
&&\leq C\left\|\|u\|_{\dot{\mathcal{N}}_{q_{0},\mu, \infty}^{s_{0}}}\right\|_{L^{\alpha_{0}, r}(0, T)}\left\|\|v\|_{\dot{\mathcal{N}}_{q_{0}, \mu,\infty}^{s_{0}}}\right\|_{L^{\alpha_{0}, r}(0, T)},
\end{eqnarray*}
namely
\begin{equation}\label{eqmax313}
\| \mathbb{P}(u \cdot \nabla v)\|_{L^{\alpha, r}(0, T;\dot{\mathcal{N}}_{q,\mu, \infty}^{s})}
\leq C\| u\|_{L^{\alpha_{0}, r}(0, T;\dot{\mathcal{N}}_{q_{0},\mu, \infty}^{s_{0}})}
\| v\|_{{L^{\alpha_{0}, r}(0, T;\dot{\mathcal{N}}_{q_{0},\mu, \infty}^{s_{0}}})},
\end{equation}
where $C=C\left(n,\mu, q, \alpha, q_{0}, \alpha_{0}, r\right)$
is a constant independent of $0<T \leq \infty$.
Since $p \leq q, (n-\mu) / p-1<s_{0}$,
and since $2 / \alpha_{0}+(n-\mu) / q_{0}-s_{0}=1,$ we have
\begin{eqnarray*}
(n-\mu) / q \leq (n-\mu) / p<s_{0}+1=2 / \alpha_{0}+(n-\mu) / q_{0},\\
k=2+(n-\mu) / p-(2 / \alpha+(n-\mu) / q-s)=-1+(n-\mu) / p.
\end{eqnarray*}
Hence it follows from Lemma \ref{lemmax31} with $k=-1+(n-\mu) / p$ that
\begin{eqnarray}
&&\|u\|_{L^{\alpha_{0}, r}(0, T ; \dot{\mathcal{N}}_{q_{0},\mu, \infty}^{s_{0}})}\nonumber\\
&&\leq C\left(\left\|e^{-t A} a\right\|_{L^{\alpha_{0}, r}(0, T ; \dot{\mathcal{N}}_{q_{0},\mu, 1}^{s_{0}})}\right.
+\left\|u_{t}\right\|_{L^{\alpha, r}(0, T ; \dot{\mathcal{N}}_{q,\mu, \infty}^{s})}
+\|A u\|_{L^{\alpha, r}(0, T ; \dot{\mathcal{N}}_{q,\mu, \infty}^{s})} ),\label{eqmax314}\\
&&\|v\|_{L^{\alpha_{0}, r}(0, T ; \dot{B}_{q_{0},\mu, \infty}^{s_{0}})}\nonumber\\
&&\leq C\left(\left\|e^{-t A} b\right\|_{L^{\alpha_{0}, r}(0, T ; \dot{\mathcal{N}}_{q_{0}, \mu, 1}^{s_{0}})}\right.
+\left\|v_{t}\right\|_{L^{\alpha, r}(0, T ; \dot{\mathcal{N}}_{p, \mu, \infty}^{s})}
+\|A v\|_{L^{\alpha, r}(0, T ; \dot{\mathcal{N}}_{p,\mu, \infty}^{s})} )\label{eqmax315}
\end{eqnarray}
with $C=C\left(n,\mu, q, \alpha, q_{0}, \alpha_{0}, r \right)$ independent of $0<T \leq \infty$.
Now from (\ref{eqmax313}), (\ref{eqmax314}) and (\ref{eqmax315}), we obtain the desired estimate (\ref{eqmax306}). This proves Lemma \ref{lemmax32}.
\end{proof}

The following lemma
is useful to prove
the additional
regularity of $u.$
It states that if $u$ itself belongs to certain scaling invariant class, then we
gain the regularity of the nonlinear term $u\cdot \nabla v $ in accordance with that of $v$.

\begin{lem}\label{lemmax33}
Let $u \in L^{\alpha_{0}, r}(0, T ; \dot{\mathcal{N}}_{q_{0}, \mu_{0}, \infty}^{s_{0}})$,
$0<T \leq \infty$ for $2 / \alpha_{0}+(n-\mu_{0}) / q_{0}-s_{0}=1$ with $s_{0}>0$.
Suppose that $q^{*}, \alpha^{*}, s^{*}, \mu^{*}$ and $p^{*}$ satisfy
\begin{eqnarray}\label{eqmax316}
(n-\mu^{*})/ q^{*} \leq (n-\mu^{*}) / p^{*}<2 / \alpha^{*}+(n-\mu^{*})/ q^{*}, \nonumber\\
\quad 1<\alpha^{*}<\alpha_{0},
\quad-1<s^{*}<s_{0}-1,\quad 0\leq\mu^{*}\leq\mu_{0}<n.
\end{eqnarray}
Then for every function $v$ with $v_{t},$ $Av\in L^{\alpha^{*}, r^{*}}(0, T ; \dot{\mathcal{N}}_{q^{*},\mu^{*}, \infty}^{s^{*}}),$ $0<T \leq \infty$ for $1 \leq r^{*} \leq \infty,$ and $v(0) \in \mathcal{\dot{\mathcal{N}}}_{p^{*},\mu^{*}, r^{*}}^{k^{*}}$ for $k^{*}=2+(n-\mu^{*}) / p^{*}-(2 / \alpha^{*}+(n-\mu^{*} )/q^{*}-s^{*})$, it holds that
\begin{equation*}
\mathbb{P}(u \cdot \nabla v) \in L^{\alpha^{*}, r^{*}}(0, T ; \dot{\mathcal{N}}_{q^{*}, \mu^{*} \infty}^{s^{*}})
\end{equation*}
with the estimate
\begin{eqnarray}\label{eqmax317}
&&\|\mathbb{P}(u \cdot \nabla v)\|_{L^{\alpha^{*}, r^{*}}(0, T ; \dot{\mathcal{N}}_{q^{*},\mu^{*}, \infty}^{s^{*}})}\nonumber\\
&&\leq C\|u\|_{L^{\alpha_{0}, r}(0, T ; \dot{\mathcal{N}}_{q_{0},\mu_{0}, \infty}^{s_{0}})}
(\|v(0)\|_{ \dot{\mathcal{N}}_{p^{*},\mu^{*}, r^{*}}^{k^{*}}}+\|v_{t}\|_{L^{\alpha^{*}, r^{*}}(0, T ; \dot{\mathcal{N}}_{q^{*}, \mu^{*}, \infty}^{s^{*}})}\nonumber\\
&&+\|A v\|_{L^{\alpha^{*}, r^{*}}(0, T ; \dot{\mathcal{N}}_{q^{*},\mu^{*}, \infty}^{s^{*}})} )
\end{eqnarray}
where $C=C\left(n,\mu_{0}, \mu^{*}, q_{0}, \alpha_{0}, q^{*}, \alpha^{*}, s^{*}, p^{*}, r^{*}\right)$ is a constant independent of $0<T \leq \infty$.
\end{lem}

\begin{proof}
Let us take $\alpha_{0}^{*}>\alpha^{*}, q_{0}^{*} \geq q^{*}$ and $s_{0}^{*}>s^{*}+1$ so that
\begin{equation}\label{eqmax318}
2 / \alpha_{0}^{*}+(n-\mu^{*}) / q_{0}^{*}-s_{0}^{*}=2 / \alpha^{*}+(n-\mu^{*})/ q^{*}-s^{*}-2,
\end{equation}
\begin{equation}\label{eqmax319}
1 / \alpha_{0}^{*}=1 / \alpha^{*}-1 / \alpha_{0},
\end{equation}
\begin{equation}\label{eqmax320}
(n-\mu^{*})/ p^{*}-\left(2 / \alpha^{*}+(n-\mu^{*})/ q^{*}\right)+s^{*}+2<s_{0}^{*}.
\end{equation}

Since $s^{*}+1<s_{0}^{*},$ it holds that $0<s_{0}^{*}-\left(s^{*}+1\right)$.
Since $s_{0}>0,$ there exists $\sigma_{1}>0$ such that
\begin{equation}\label{eqmax321}
-s_{0}<\sigma_{1}<s_{0}^{*}-\left(s^{*}+1\right).
\end{equation}
Since $s^{*}<s_{0}-1$ and since $s^{*}+1-s_{0}<s^{*}+1<s_{0}^{*}$
implied by $s_{0}>0,$ it holds that
\begin{equation}
0<s_{0}-\left(s^{*}+1\right), \quad-s_{0}^{*}<s_{0}-\left(s^{*}+1\right).
\end{equation}
Hence there exists $\sigma_{2}>0$ such that
\begin{equation}\label{eqmax322}
-s_{0}^{*}<\sigma_{2}<s_{0}-\left(s^{*}+1\right).
\end{equation}
Let us define $q_{1}^{*}, q_{2}^{*}, \tilde{q}_{1}^{*}$ and $\tilde{q}_{2}^{*}$ by

\begin{eqnarray*}
\frac{n-\mu^{*}}{q_{1}^{*}}-(s^{*}+1+\sigma_{1})
=\frac{n-\mu^{*}}{q_{0}^{*}}-s_{0}^{*},\\
\frac{n-\mu^{*}}{q_{2}^{*}}+\sigma_{1}
=\frac{n-\mu_{0}}{q_{0}}-s_{0},
\end{eqnarray*}
\begin{eqnarray*}
\frac{n-\mu^{*}}{\tilde{q}_{1}^{*}}+\sigma_{2}
=\frac{n-\mu^{*}}{q_{0}^{*}}-s_{0}^{*},\\
\frac{n-\mu^{*}}{\tilde{q}_{2}^{*}}
-(s^{*}+1+\sigma_{2})=\frac{n-\mu_{0}}{q_{0}}-s_{0}.
\end{eqnarray*}
By (\ref{eqmax321}) and (\ref{eqmax322}), we have
\begin{equation*}
q_{0}^{*}<q_{1}^{*};\quad q_{0}<q_{2}^{*};\quad q_{0}^{*}<\tilde{q}_{1}^{*}; \quad q_{0}<\tilde{q}_{2}^{*}.
\end{equation*}
Hence it follows from Lemma \ref{lem1}(ii) that
\begin{eqnarray}
\dot{\mathcal{N}}_{q_{0}^{*}, \mu^{*},\infty}^{s_{0}^{*}} \hookrightarrow \dot{\mathcal{N}}_{q_{1}^{*},\mu^{*}, \infty}^{s^{*}+1+\sigma_{1}},\nonumber\\
\dot{\mathcal{N}}_{q_{0}, \mu_{0}, \infty}^{s_{0}} \hookrightarrow \dot{\mathcal{N}}_{q_{2}^{*},\mu^{*}, \infty}^{-\sigma_{1}},\nonumber\\
\dot{\mathcal{N}}_{q_{0}^{*},\mu^{*}, \infty}^{s_{0}^{*}} \hookrightarrow \dot{\mathcal{N}}_{\tilde{q}_{1}^{*}, \mu^{*}, \infty}^{-\sigma_{2}},\nonumber\\
\dot{\mathcal{N}}_{q_{0},\mu_{0}, \infty}^{s_{0}} \hookrightarrow \dot{\mathcal{N}}_{\tilde{q}_{2}^{*}, \mu^{*}, \infty}^{s^{*}+1+\sigma_{2}}.\label{eqmax325}
\end{eqnarray}
Furthermore by (\ref{eqmax318}) and (\ref{eqmax319}), we have that
\begin{equation}\label{eqmax326}
\frac{1}{q_{1}^{*}}+\frac{1}{q_{2}^{*}}
=\frac{1}{q^{*}},\ \ \ \ \
\frac{1}{\tilde{q}_{1}^{*}}+\frac{1}{\tilde{q}_{2}^{*}}
=\frac{1}{\tilde{q}^{*}}.
\end{equation}
Hence it follows from (\ref{eqmax325}), (\ref{eqmax326})
and Proposition \ref{prop32} that
\begin{eqnarray}
&&\|\mathbb{P}(u \cdot \nabla v)\|_{\dot{\mathcal{N}}_{q^{*},\mu^{*}, \infty}^{s^{*}}}
=\|\nabla \cdot \mathbb{P}(u \otimes v)\|_{\dot{\mathcal{N}}_{q^{*},\mu^{*}, \infty}^{s^{*}}}
\leq C\|u \otimes v\|_{\dot{\mathcal{N}}_{q^{*},\mu^{*}, \infty}^{s^{*}+1}}\\
&&\leq C(\|v\|_{\dot{\mathcal{N}}_{q_{1}^{*},\mu^{*}, \infty}^{s^{*}+1+\sigma_{1}}}\|u\|_{\dot{\mathcal{N}}_{q_{2}^{*},\mu^{*}, \infty}^{-\sigma_{1}}}
+\|v\|_{\dot{\mathcal{N}}_{\tilde{q}_{1}^{*},\mu^{*} ,\infty}^{-\sigma_{2}}}\|u\| _{\dot{\mathcal{N}}_{\tilde{q}_{2}^{*},\mu^{*}, \infty}^{s^{*}+1+\sigma_{2}}})\\
&&\leq C\|v\|_{\dot{\mathcal{N}}_{q_{0}^{*},\mu^{*}, \infty}^{s_{0}^{*}}}\|u\|_{\dot{\mathcal{N}}_{q_{0},\mu_{0}, \infty}^{s_{0}}}
\end{eqnarray}
where $C=C\left(n,\mu_{0},\mu^{*}, q_{0}, \alpha_{0}, s_{0}, q^{*}, \alpha^{*}, s^{*}\right)$.
Since $1 / \alpha^{*}=1 / \alpha_{0}^{*}+1 / \alpha_{0}$,
implied by (\ref{eqmax319}), we have by
the H\"{o}lder inequality that
\begin{eqnarray}\label{eqmax327}
&&\left\|\|\mathbb{P}(u \cdot \nabla v)\|_{\dot{\mathcal{N}}_{q^{*},\mu^{*}, \infty}^{s^{*}}}\right\|_{L^{\alpha^{*}, r^{*}}(0, T)}\nonumber\\
&&\leq
\left\|\|v\|_{\dot{\mathcal{N}}_{q_{0}^{*}, \mu^{*},\infty}^{s_{0}^{*}}}\right\|_{L^{\alpha_{0}^{*}, r^{*}}(0, T)}
\left\|\|u\|_{\dot{\mathcal{N}}_{q_{0},\mu_{0}, \infty}^{s_{0}}}\right\|_{L^{\alpha_{0}, \infty}(0, T)}\nonumber\\
&&\leq\|\|v\|_{\dot{\mathcal{N}}_{q_{0}^{*},\mu^{*}, \infty}^{s_{0}^{*}}}\|_{L^{\alpha_{0}^{*}, r^{*}}(0, T)}
\|\|u\|_{\dot{\mathcal{N}}_{p_{0},\mu_{0}, \infty}^{s_{0}}}\|_{L^{\alpha_{0}, r}(0, T)}
\end{eqnarray}
where $C=C\left(n,\mu_{0},\mu^{*}, q_{0}, \alpha_{0}, q^{*}, \alpha^{*}, s^{*}, p^{*}, r^{*}\right)$ is a constant independent of $0<T \leq \infty$.
By (\ref{eqmax318}) and
(\ref{eqmax320}), we have that
$(n-\mu^{*}) / q^{*} \leq (n-\mu^{*}) / p^{*}<2 / \alpha_{0}^{*}+(n-\mu^{*}) / q_{0}^{*}$
and hence it follows from Lemma \ref{lemmax31} that
\begin{eqnarray}\label{eqmax328}
&&\|v\|_{L^{\alpha_{0}^{*}, r^{*}} (0, T ; \dot{\mathcal{N}}_{q_{0}^{*},\mu^{*}, \infty}^{s_{0}^{*}})}\\
&&\leq C(\|v(0)\|_{\dot{\mathcal{N}}_{p^{*},\mu^{*}, r^{*}}^{k^{*}}}+\|v_{t}\|_{L^{\alpha^{*}, r^{*}}(0, T ; \dot{\mathcal{N}}_{q^{*},\mu^{*}, \infty}^{s^{*}})}
+\|A v\|_{L^{\alpha^{*}, r^{*}}(0, T ; \dot{\mathcal{N}}_{q^{*},\mu^{*}, \infty}^{s^{*}})}).
\end{eqnarray}
with a constant $C$ independent of $0<T \leq \infty$.
Now from (\ref{eqmax327}) and (\ref{eqmax328}), we obtain the
desired estimate (\ref{eqmax317}). This proves Lemma \ref{lemmax33}.
\end{proof}

\section{ Proof of Theorems \ref{thm2}, \ref{thm3}, \ref{thm4}}

\subsection{ Proof of Theorem \ref{thm2}}
(i)\ \ In case $1 \leq q<\infty$:
We first prove existence of the solution $u$ of (\ref{eqmaxns}) on $(0, T^{*})$ for some
$0<T_{*} \leq T$.
Let $a \in \dot{\mathcal{N}}_{p,\mu, r}^{-1+\frac{n-\mu}{p}}$ and $f \in L^{\alpha, r}(0, T ; \dot{\mathcal{N}}_{q,\mu, \infty}^{s})$ for $2 / \alpha+(n-\mu) / q-s=3$ with $1<q<\infty,$
$1<\alpha<\infty$, $0\leq\mu<n$ and $-1<s$ satisfying $(\ref{eqmax102}).$ We define $u_{0}$ by
\begin{equation}
u_{0}(t)=e^{-t A} a+\int_{0}^{t} e^{-(t-\tau) A} \mathbb{P}f(\tau) d \tau \equiv u_{0}^{(1)}(t)+u_{0}^{(2)}(t), \quad 0<t<T.
\end{equation}

We shall solve (\ref{eqmaxns}) in the form $u=u_{0}+v$.
Then the solvability of (\ref{eqmaxns}) can be reduced to
construct the solution $v$ of the following equation
\begin{equation}\label{ns1}
\left\{\begin{array}{l}{\frac{d v}{d t}+A v=-\mathbb{P}\left(u_{0} \cdot \nabla v+v \cdot \nabla u_{0}+v \cdot \nabla v+u_{0} \cdot \nabla u_{0}\right) \quad \text { a.e. } t \in\left(0, T_{*}\right) \text { in } \dot{\mathcal{N}}_{q,\mu, \infty}^{s}} \\ {v(0)=0.}\end{array}\right.
\end{equation}

By assumption (\ref{eqmax102}) we have that
\begin{equation}
(n-\mu)/ q \leq (n-\mu) / p<2 / \alpha+(n-\mu) / q.
\end{equation}

Hence it follows from Proposition \ref{propmax201} and Theorem \ref{thm1} that
\begin{equation}\label{eqmax401}
\left\|\frac{d u_{0}^{(1)}}{d t}\right\|_{L^{\alpha, r}(0, \infty ; \dot{\mathcal{N}}_{q,\mu, 1}^{s})}
+\left\|A u_{0}^{(1)}\right\|_{L^{\alpha, r}(0, \infty ; \dot{\mathcal{N}}_{q,\mu, 1}^{s})}
\leq C\|a\|_{\dot{\mathcal{N}}_{p,\mu, r}^{-1+(n-\mu)/p}},
\end{equation}

\begin{equation}\label{eqmax402}
\left\|\frac{d u_{0}^{(2)}}{d t}\right\|_{L^{\alpha, r}(0, T ; \dot{\mathcal{N}}_{q,\mu, \infty}^{s})}
+\left\|A u_{0}^{(2)}\right\|_{L^{\alpha, r}(0, T ; \dot{\mathcal{N}}_{q,\mu, \infty}^{s})}
\leq C\|f\|_{L^{\alpha, r}(0, T ; \dot{\mathcal{N}}_{q,\mu, \infty}^{s})},
\end{equation}
where $C=C(n,\mu, p, \alpha, r, q)$ is a constant independent of $0<T \leq \infty$.
We solve (\ref{ns1}) by the following successive approximation
\begin{equation}\label{ns10}
\left\{\begin{array}{ll}{\frac{d v_{0}}{d t}+A v_{0}=-\mathbb{P}\left(u_{0} \cdot \nabla u_{0}\right)} & {\text { a.e. } t \in(0, T) \text { in } \dot{\mathcal{N}}_{q,\mu, \infty}^{s}} \\ {v_{0}(0)=0.}\end{array}\right.
\end{equation}
\begin{equation}\label{ns1j}
\left\{\begin{array}{l}{\frac{d v_{j+1}}{d t}+A v_{j+1}} \\ {=-\mathbb{P}\left(u_{0} \cdot \nabla v_{j}+v_{j} \cdot \nabla u_{0}+v_{j} \cdot \nabla v_{j}+u_{0} \cdot \nabla u_{0}\right) \quad \text { a.e. } t \in(0, T) \text { in } \dot{\mathcal{N}}_{q,\mu, \infty}^{s}} \\ {v_{j+1}(0)=0, \quad j=0,1, \cdots.}\end{array}\right.
\end{equation}
Set
\begin{equation*}
X_{T}=\left\{v : \mathbb{R}^{n} \times(0, T) \rightarrow \mathbb{R}^{n} ; v_{t}, A v \in L^{\alpha, r}(0, T ; \dot{\mathcal{N}}_{q,\mu, \infty}^{s}), v(0)=0\right\}
\end{equation*}
equipped with the norm
\begin{equation*}
\|v\|_{X_{T}} :=\left\|v_{t}\right\|_{L^{\alpha, r}(0, T ; \dot{\mathcal{N}}_{q,\mu, \infty}^{s})}+\|A v\|_{L^{\alpha, r}(0, T ; \dot{\mathcal{N}}_{q,\mu, \infty}^{s})}.
\end{equation*}
$X_{T}$ is a Banach space.
By (\ref{eqmax401}) and (\ref{eqmax402}), it holds that
$\left\|u_{0}\right\|_{X_{T}}<\infty .$ Since $a \in \dot{\mathcal{N}}_{p,\mu, r}^{-1+\frac{n-\mu}{p}},$ we
have by Lemma \ref{lemmax32} that
 \begin{equation*}
-\mathbb{P}\left(u_{0} \cdot \nabla u_{0}\right) \in L^{\alpha, r}(0, T ; \dot{\mathcal{N}}_{q, \mu, \infty}^{s})
\end{equation*}
with the estimate
\begin{equation*}
\left\|\mathbb{P}\left(u_{0} \cdot \nabla u_{0}\right)\right\|_{L^{\alpha, r}(0, T ; \dot{\mathcal{N}}_{q,\mu, \infty}^{s})}
\leq C\left(\left\|e^{-t A} a\right\|_{L^{\alpha_{0}, r}(0, T ; \dot{\mathcal{N}}_{q_{0},\mu, 1}^{s_{0}})}
+\left\|u_{0}\right\|_{X_{T}}\right)^{2},
\end{equation*}
where $C = C(n,\mu, q, \alpha,\alpha_{0}, q_{0},p, r)$ is a constant independent of $0 < T \leq\infty$.
Hence it follows from
Theorem \ref{thm1} that there is a unique solution $v_{0}$ of
(\ref{ns10}) in the class  $X_{T}$.

Assume that $v_{j} \in X_{T}$. Again by Lemma \ref{lemmax32}, we have that
\begin{equation*}
-\mathbb{P}\left(u_{0} \cdot \nabla v_{j}+v_{j} \cdot \nabla u_{0}+v_{j} \cdot \nabla v_{j}+u_{0} \cdot \nabla u_{0}\right) \in L^{\alpha, r}(0, T ; \dot{\mathcal{N}}_{q,\mu, \infty}^{s})
\end{equation*}
with the estimate
\begin{eqnarray}\label{eqmax404}
&&\left\|\mathbb{P}\left(u_{0} \cdot \nabla v_{j}+v_{j} \cdot \nabla u_{0}+v_{j} \cdot \nabla v_{j}+u_{0} \cdot \nabla u_{0}\right)\right\|_{L^{\alpha, r}(0, T ; \dot{\mathcal{N}}_{q,\mu, \infty}^{s})}\nonumber\\
&&\leq C\bigg(2(\|e^{-t A} a\|_{L^{\alpha_{0}, r}(0, T ; \dot{\mathcal{N}}_{q_{0}, \mu, 1}^{s_{0}})}
+\|u_{0}\|_{X_{T}})\|v_{j}\|_{X_{T}}\nonumber\\
&&+\|v_{j}\|_{X_{T}}^{2}+
(\|e^{-t A} a\|_{L^{\alpha_{0}, r}(0, T ; \dot{\mathcal{N}}_{q_{0}, \mu, 1}^{s_{0}})}+\|u_{0}\|_{X_{T}})^{2}\bigg)
\end{eqnarray}
for some $\alpha_{0}>\alpha, q_{0}\geq q$ and $s_{0}>s$ such that
$2/\alpha_{0}+(n-\mu)/q_{0}-s_{0}=1,$
where $C=C\left(n,\mu, q, \alpha, q_{0}, \alpha_{0}, p, r \right)$ is a constant independent of $0<T \leq \infty$.
Then it follows from Theorem \ref{thm1}
 that there exists a unique solution $v_{j+1}$ of
 (\ref{ns1j}) in $X_{T}$ with the estimate
\begin{eqnarray}\label{eqmax404}
\left\|v_{j+1}\right\|_{X_{T}}
&\leq& C\bigg(
2(\|e^{-t A} a\|_{L^{\alpha_{0}, r}(0, T ; \dot{\mathcal{N}}_{q_{0},\mu, 1}^{s_{0}})}+\|u_{0}\|_{X_{T}})\|v_{j}\|_{X_{T}}\nonumber\\
&+&\|v_{j}\|_{X_{T}}^{2}
+(\|e^{-t A} a\|_{L^{\alpha_{0}, r}(0, T ; \dot{\mathcal{N}}_{q_{0},\mu, 1}^{s_{0}})}+\|u_{0}\|_{X_{T}})^{2}
\bigg)
\end{eqnarray}
where $C=C\left(n,\mu, q, \alpha, q_{0}, \alpha_{0}, p, r \right)$ is a constant independent of $0<T \leq \infty$.

By induction, we have that
$v_{j} \in X_{T}$ for all $j=1,2, \cdots$.
Hence, if there is
$0<T_{*} \leq T$
such that
\begin{equation}\label{eqmax405}
\left\|e^{-t A} a\right\|_{L^{\alpha_{0}, r}(0, T_{*} ; \dot{\mathcal{N}}_{q_{0},\mu, 1}^{s_{0}})}+\left\|u_{0}\right\|_{X_{T_{*}}}<\frac{1}{4 C},
\end{equation}
then we obtain from (\ref{eqmax404}) that
\begin{eqnarray}\label{eqmax406}
&&\left\|v_{j}\right\|_{X_{T_{*}}}\nonumber\\
&&\leq
\frac{1}{2 C}\bigg(1-2 C(\|e^{-t A} a\|_{L^{\alpha_{0}, r}(0, T_{*} ; \dot{\mathcal{N}}_{q_{0}, \mu,1}^{s_{0}})}+\|u_{0}\|_{X_{T_{*}}})\nonumber\\
&&-\sqrt{1-4 C(\|e^{-t A} a\|_{L^{\alpha_{0}, r}(0, T_{*} ; \dot{\mathcal{N}}_{q_{0},\mu, 1}^{s_{0}})}+\|u_{0}\| _{X_{T_{*}}})}\bigg)
\equiv K,\ \ for \ \ all \ \ j=1,2,\cdots.
\end{eqnarray}
It should be noted that all constants $C$ appearing in
(\ref{eqmax404}), (\ref{eqmax405}) and (\ref{eqmax406}) are the same and independent of $0< T\leq \infty.$

Defining $w_{j} \equiv v_{j}-v_{j-1}\left(v_{-1}=0\right)$,
we obtain from $(\ref{ns1j})$ that
\begin{equation*}
\left\{\begin{array}{l}{\frac{d w_{j+1}}{d t}+A w_{j+1}=-\mathbb{P}\left(u_{0} \cdot \nabla w_{j}+w_{j} \cdot \nabla u_{0}+v_{j} \cdot \nabla w_{j}+w_{j} \cdot \nabla v_{j-1}\right),}
 \\ {w_{j+1}(0)=0 ,\quad j=0,1, \cdots.}\end{array}\right.
\end{equation*}
Similarly to (\ref{eqmax404}), we have by (\ref{eqmax406}) that
\begin{eqnarray}\label{eqmax407}
&&\left\|w_{j+1}\right\|_{X_{T_{*}}}\nonumber\\
&&\leq C\bigg(2(\|e^{-t A} a\|_{L^{\alpha_{0}, r}(0, T_{*} ; \dot{\mathcal{N}}_{q_{0},\mu, 1}^{s _{0}})}+\|u_{0}\|_{X_{T_{*}}})\|w_{j}\|_{X_{T_{*}}}\nonumber\\
&&+\|v_{j}\|_{X_{T_{*}}}\|w_{j}\|_{X_{T_{*}}}
+\|v_{j-1}\|_{X_{T_{*}}}\|w_{j}\|_{X_{T_{*}}}\bigg)\nonumber\\
&&\leq C\bigg(2(\|e^{-t A} a\|_{L^{\alpha_{0}, r}(0, T_{*} ; \dot{\mathcal{N}}_{q_{0}, \mu, 1}^{s_{0}})}+\|u_{0}\|_{X_{T_{*}}})+2K\bigg)\|w_{j}\|_{X_{T_{*}}}, j=1,2,\cdots,
\end{eqnarray}
which yields that
\begin{equation*}
\left\|w_{j}\right\|_{X_{T_{*}}}
\leq
\left\{2C(\left\|e^{-t A} a\right\|_{L^{\alpha_{0}, r}\left(0, T_{*} ; \dot{\mathcal{N}}_{q_{0},\mu, 1}^{s_{0}}\right)}
+\left\|u_{0}\right\|_{X_{T_{*}}}
+K)\right\}^{j}\left\|v_{0}\right\|_{X_{T_{*}}},j=1,2,\cdots.
\end{equation*}
By (\ref{eqmax406}), it holds that
\begin{eqnarray*}
&&2 C\left(\left\|e^{-t A} a\right\|_{L^{\alpha_{0}, r}(0, T_{*}; \dot{\mathcal{N}}_{q_{0},\mu, 1}^{s_{0}})}+\left\|u_{0}\right\|_{X_{T_{*}}}+K\right)\\
&&=1-\sqrt{1-4 C(\|e^{-t A} a\|_{L^{\alpha_{0}, r}(0, T_{*} ; \dot{\mathcal{N}}_{q_{0}, \mu, 1}^{s_{0}})}+\|u_{0}\|_{X_{T_{*}}})}<1,
\end{eqnarray*}
from which it follows that
\begin{equation*}
\sum_{j=1}^{\infty}\left\|w_{j}\right\|_{X_{T_{*}}}<\infty.
\end{equation*}
This implies that there exists a limit
$v \in X_{T_{*}}$  of $\{v_{j}\}_{j=1}^{\infty}$  in $ X_{T_{*}}.$
Now letting $j \rightarrow \infty$ in both sides of $(\ref{ns1j})$,
we see easily from Lemma \ref{lemmax32} that $v$ is a solution of $(\ref{ns1})$ on $(0, T_{*})$,
provided the hypothesis (\ref{eqmax405}) is fulfilled.
Concerning the validity of the condition  (\ref{eqmax405}), since
$1 \leq r<\infty,$  we see from (\ref{eqmax302}), (\ref{eqmax401}) and (\ref{eqmax402}) that there exists $0 < T_{*}\leq T$ such that the estimate (\ref{eqmax405}) holds. Hence we have proved the
existence of the solution $u$ of (\ref{eqmaxns}) on $(0, T_{*})$.

Next, we shall show (\ref{eqmax105}). Indeed, for $\alpha<\alpha_{0}, q \leq q_{0}$
and $\max \{s, (n-\mu) / p-1\}<s_{0}$ satisfying
$2 / \alpha_{0}+(n-\mu) / q_{0}-s_{0}=1,$ it holds that
\begin{equation*}
\begin{array}{l}{(n-\mu) / q \leq (n-\mu) / p<s_{0}+1=2 / \alpha_{0}+(n-\mu) / q_{0}}, \\ {k=2+(n-\mu) / p-(2 / \alpha+(n-\mu) / q-s)=-1+(n-\mu) / p}.\end{array}
\end{equation*}
Hence, from Lemma \ref{lemmax31} we obtain (\ref{eqmax105}).

Now, it remains to prove the uniqueness.
Let $u_{1}$ and $u_{2}$ be two solutions of $(\ref{eqmaxns})$ on $\left(0, T_{*}\right)$
in the class $(\ref{eqmax107}).$ Defining $w=u_{1}-u_{2},$ we have
\begin{equation}\label{eqmax408}
\left\{\begin{array}{l}{\frac{d w}{d t}+A w=-\mathbb{P}\left(u_{1} \cdot \nabla w+w \cdot \nabla u_{2}\right) \quad \text { a.e. } t \in\left(0, T_{*}\right) \text { in } \dot{\mathcal{N}}_{q,\mu, \infty}^{s}} \\ {w(0)=0.}\end{array}\right.
\end{equation}
It holds by Lemma \ref{lemmax32} that

\begin{eqnarray}\label{eqmax409}
&&\left\|\mathbb{P}\left(u_{1} \cdot \nabla w+w \cdot \nabla u_{2}\right)\right\|_{L^{\alpha, r}(0, T_{*} ; \dot{\mathcal{N}}_{q,\mu, \infty}^{s})}\nonumber\\
&&\leq C(\|e^{-t A} a\|_{L^{\alpha_{0}, r}(0, T_{*}; \dot{\mathcal{N}}_{q_{0},\mu_{0}, 1}^{s_{0}})}+\|u_{1}\|_{X_{T_{*}}}+\|u_{2}\|_{X_{T_{*}}})\|w\|_{ X_{T_{*}}}.
\end{eqnarray}
Then it follows from (\ref{eqmax408}), (\ref{eqmax409}) and Theorem \ref{thm1} that
\begin{equation}\label{eqmax410}
\|w\|_{X_{T_{*}}}
\leq C\left(\left\|e^{-t A} a\right\|_{L^{\alpha_{0}, r}(0, T_{*}; \dot{\mathcal{N}}_{q_{0},\mu, 1}^{s_{0}})}+\|u_{1}\|_{X_{T_{*}}}+\|u_{2}\|_{X_{T_{*}}}\right)\|w\|_{X_{T_{*}}},
\end{equation}
where $ C = C(n,\mu,  q, \alpha, q_{0}, \alpha_{0}, p, r)$ is a constant independent of $T_{*}$.
Since $1 \leq r<\infty,$  there  is some $\sigma>0$ such that
\begin{equation}\label{eqmax411}
\left\|e^{-t A} a\right\|_{L^{\alpha_{0}, r}(\tau, \tau+\sigma ; \dot{\mathcal{N}}_{q_{0},\mu, 1}^{s_{0}})}<\frac{1}{4 C},
\end{equation}
\begin{equation}\label{eqmax412}
\left\|\frac{d u_{i}}{d t}\right\|_{L^{\alpha, r}(\tau, \tau+\sigma ; \dot{\mathcal{N}}_{q,\mu, \infty}^{s})}+\left\|A u_{i}\right\|_{L^{\alpha, r}(\tau, \tau+\sigma ; \dot{\mathcal{N}}_{q,\mu, \infty}^{s})}<\frac{1}{4 C},\quad i=1 ,2
\end{equation}
hold for all $\tau \in\left(0, T_{*}\right)$, where $C$ is the same constant as in (\ref{eqmax410}).
Then we have by (\ref{eqmax410}), (\ref{eqmax411}) that $w(t)=0$ for $t \in[0, \sigma].$
Since $\sigma$ depends only on the constant $C$ in (\ref{eqmax410}), we
may repeat the same argument on the interval $[\sigma, 2 \sigma]$ to conclude that $w(t)=0$ for $t \in[0,2 \sigma].$
Proceeding further for $t \geq 2 \sigma$ successively, within finitely many steps, we deduce that $w(t)=0$ for all $t \in[0, T_{*}),$ which implies the desired uniqueness.

(ii)\ \ In case $q=\infty$.
Let $a \in \dot{\mathcal{N}}_{p,\mu, \infty}^{-1+\frac{(n-\mu)}{p}}$
and $f \in L^{\alpha, \infty}(0, T ; \dot{\mathcal{N}}_{q,\mu, \infty}^{s})$ for $2 / \alpha+(n-\mu) / q-s=3$ with $n-\mu<p \leq q<\infty, 1<\alpha<\infty$ and $-1<s$ satisfying $(\ref{eqmax102}).$
Although the estimates (\ref{eqmax302}), (\ref{eqmax401})
and (\ref{eqmax402}) hold even for $q=\infty,$ we need such an additional assumption as $(\ref{eqmax106})$ to deduce $(\ref{eqmax405})$. Indeed, by $(\ref{eqmax106})$ it holds that
\begin{eqnarray}
a=a_{0}+a_{1}, \ \ \text { for  } \ \ a_{0} \in \mathcal{M}_{p,\mu} \text { and } a_{1} \in \dot{\mathcal{N}}_{p,\mu, \infty}^{-1+\frac{(n-\mu)}{p}} \text { with }\left\|a_{1}\right\|_{\dot{\mathcal{N}}_{p,\mu, \infty}^{-1+\frac{(n-\mu)}{p}}}<C \eta,\label{eqmax413}\\
\|f(t)\|_{\dot{\mathcal{N}}_{q,\mu, \infty}^{s}}=g_{0}(t)+g_{1}(t), \quad 0<t<T \nonumber\\
\text { for  }
g_{0} \in L^{\infty}(0, T) \text { and } g_{1} \in L^{\alpha, \infty}(0, T) \text { with }\left\|g_{1}\right\|_{L^{\alpha, \infty}(0, T)}<C \eta,\label{eqmax414}
\end{eqnarray}
where $C=C(n,\mu, q, \alpha)$ is independent of $T$.
For a moment, let us assume (\ref{eqmax413}) and (\ref{eqmax414}).
 First, by (\ref{eqmax307}), notice that we may take
$s_{0}>0$ in $(\ref{eqmax405})$ because $s>-1 .$ since $n-\mu<p \leq q \leq q_{0}$ and
since $2 / \alpha_{0}+(n-\mu) / q_{0}-s_{0}=1,$ we have
\begin{equation*}
\|e^{-t A} a_{0}\|_{\dot{\mathcal{N}}_{q_{0},\mu, 1}^{s_{0}}}
\leq C t^{-\frac{1}{2}(\frac{(n-\mu)}{p}-\frac{(n-\mu)}{q_{0}})-\frac{s_{0}}{2}}
\|a_{0}\|_{\dot{\mathcal{N}}_{p,\mu, \infty}^{0}} \leq C t^{-\frac{1}{\alpha_{0}}+\frac{1}{2}(1-\frac{(n-\mu)}{p})}
\|a_{0}\|_{\mathcal{M}_{p,\mu}},
\end{equation*}
which yields that
\begin{equation}\label{eqmax415}
\left\|e^{-t A} a_{0}\right\|_{L^{\alpha_{0}, \infty}(0, T ; \dot{\mathcal{N}}_{q_{0},\mu, 1}^{s_{0}})}
 \leq C\left(\int_{0}^{T}\left\|e^{-t A} a_{0}\right\|_{\dot{\mathcal{N}}_{q_{0},\mu, 1}^{s_{0}}}^{\alpha_{0}} d t\right)^{\frac{1}{\alpha_{0}}}
  \leq C\left\|a_{0}\right\|_{\mathcal{M}_{p,\mu}}T^{\frac{1}{2}(1-\frac{(n-\mu)}{p})}
\end{equation}
with $C=C\left(n,\mu, q_{0}, \alpha_{0}, p\right)$ independent of $T.$
Similarly, since
 $2 / \alpha+(n-\mu) / q-s=3,$ we have
\begin{eqnarray*}
\left\|A e^{-t A} a_{0}\right\|_{\dot{\mathcal{N}}_{q,\mu, 1}^{s}}
\leq\left\|e^{-t A} a_{0}\right\|_{\dot{\mathcal{N}}_{q,\mu, 1}^{s+2}}
&\leq& C t^{-\frac{1}{2}(\frac{(n-\mu)}{p}-\frac{(n-\mu)}{q})-\frac{s+2}{2}}
\left\|a_{0}\right\|_{\dot{\mathcal{N}}_{p,\mu, \infty}^{0}} \\
&\leq& C t^{-\frac{1}{\alpha}+\frac{1}{2}(1-\frac{(n-\mu)}{p})}
\left\|a_{0}\right\|_{\mathcal{M}_{p,\mu}},
\end{eqnarray*}

which yields that
\begin{equation}\label{eqmax416}
\left\|A e^{-t A} a_{0}\right\|_{L^{\alpha, \infty}(0, T ; \dot{\mathcal{N}}_{q,\mu, 1}^{s})} \leq C \bigg(\int_{0}^{T}\left\|A e^{-t A} a_{0}\right\|_{\dot{\mathcal{N}}_{q,\mu, 1}^{s}}^{\alpha} d t \bigg)^{\frac{1}{\alpha}} \leq C\left\|a_{0}\right\|_{\mathcal{M}_{p,\mu}} T^{\frac{1}{2}(1-\frac{(n-\mu)}{p})}
\end{equation}
with $C=C(n,\mu, q, \alpha, p)$ independent of $T.$
By (\ref{eqmax413}), (\ref{eqmax302}) and Proposition \ref{propmax201}, it holds that
\begin{equation}
\left\|e^{-t A} a_{1}\right\|_{L^{\alpha_{0}, \infty}(0, T ; \dot{\mathcal{N}}_{q_{0},\mu, 1}^{s_{0}})}+\left\|A e^{-t A} a_{1}\right\|_{L^{\alpha, \infty}(0, T ; \dot{\mathcal{N}}_{q,\mu, 1}^{s})} \leq C \eta
\end{equation}
with $C=C(n,\mu, q, \alpha, q_{0}, \alpha_{0} )$ independent of $T$.
Similarly to $(\ref{eqmax402}),$ we obtain from $(\ref{eqmax414})$ that
\begin{equation}\label{eqmax418}
\left\|u_{0}^{(2)}\right\|_{X_{T}} \leq C\left\|g_{0}+g_{1}\right\|_{L^{\alpha, \infty}(0, T)} \leq C\left(\left\|g_{0}\right\|_{L^{\infty}(0, T)} T^{\frac{1}{\alpha}}+\eta\right)
\end{equation}
with $C=C(n,\mu, q, \alpha)$ independent of $T$,
where $u_{0}^{(2)}(t)=\int_{0}^{t} e^{-(t-\tau) A} \mathbb{P}f(\tau) d \tau.$
It follows from (\ref{eqmax415})-(\ref{eqmax418}) that
\begin{equation}\label{eqmax419}
\left\|e^{-t A} a\right\|_{L^{\alpha_{0}, \infty}(0, T ; \dot{\mathcal{N}}_{q_{0}, \mu, 1}^{s_{0}})}+\left\|u_{0}\right\|_{X_{T}} \leq C\left(\left\|a_{0}\right\|_{\mathcal{M}_{p,\mu}} T^{\frac{1}{2}\left(1-\frac{(n-\mu)}{p}\right)}+\left\|g_{0}\right\|_{L^{\infty}(0, T)} T^{\frac{1}{\alpha}}+\eta\right),
\end{equation}
where $C=C(n,\mu, q, \alpha, q_{0}, \alpha_{0}, p )$ is independent of $T$.
Hence, taking $\eta$ and $T_{*}$ sufficiently small,
we obtain from the above estimate that
\begin{equation}
\left\|e^{-t A} a\right\|_{L^{\alpha_{0}, \infty}(0, T_{*} ; \dot{\mathcal{N}}_{q_{0},\mu, 1}^{s_{0}})}+\|u_{0}\|_{X_{T_{*}}}<\frac{1}{4 C},
\end{equation}
which implies that the condition (\ref{eqmax405}) is fulfilled.

Now, we are in a position to show the decompositions (\ref{eqmax413}) and (\ref{eqmax414}). We define $a_{0}$ and
$a_{1}$ by
\begin{equation}
a_{0} \equiv \sum_{j=-\infty}^{N} \varphi_{j} * a \quad \text { and } \quad a_{1} \equiv \sum_{j=N+1}^{\infty} \varphi_{j} * a
\end{equation}
respectively. Indeed, since $1-(n-\mu) / p>0,$ implied by $(n-\mu)<p,$ it holds that
\begin{eqnarray}\label{eqmax420}
\left\|a_{0}\right\|_{\mathcal{M}_{p,\mu}}
&\leq& \sum_{j=-\infty}^{N}\left\|\varphi_{j} * a\right\|_{\mathcal{M}_{p,\mu}}\nonumber\\
&\leq&\left(\sup _{-\infty<j \leq N} 2^{\left(-1+\frac{(n-\mu)}{p}\right) j}\left\|\varphi_{j} * a\right\|_{\mathcal{M}_{p,\mu}}\right) \sum_{j=-\infty}^{N} 2^{\left(1-\frac{(n-\mu)}{p}\right) j}\nonumber\\
&\leq& C\|a\|_{\dot{\mathcal{N}}_{p,\mu, \infty}^{-1+\frac{(n-\mu)}{p}}}.
\end{eqnarray}
By (\ref{eqmax106}) we have that
\begin{eqnarray}\label{eqmax421}
\left\|a_{1}\right\|_{\dot{\mathcal{N}}_{p, \mu, \infty}^{-1+\frac{(n-\mu)}{p}}}
&=&\sup _{j \in \mathbb{Z}} 2^{(-1+\frac{n-\mu}{p}) j}\|\varphi_{j} * a_{1}\|_{\mathcal{M}_{p,\mu}}\nonumber\\
&=&\sup _{j \geq N} 2^{(-1+\frac{(n-\mu)}{p}) j}\|\varphi_{j} * \sum_{l=N+1}^{\infty} \varphi_{l} * a\|_{\mathcal{M}_{p,\mu}}\nonumber\\
&=&\sup _{j \geq N} 2^{(-1+\frac{(n-\mu)}{p}) j}\|\varphi_{j} *(\varphi_{j-1}+\varphi_{j}+\varphi_{j+1}) * a\|_{\mathcal{M}_{p,\mu}}\nonumber\\
&\leq& C \quad \underset{j \geq N}{\sup } 2^{(-1+\frac{(n-\mu)}{p}) j}\|\varphi_{j} * a\|_{\mathcal{M}_{p,\mu}}
\leq C \eta.
\end{eqnarray}
Hence from (\ref{eqmax420}) and (\ref{eqmax421}) we obtain (\ref{eqmax413}).

We next show  (\ref{eqmax414}). By  (\ref{eqmax106}), there is $ R > 0 $ such that
\begin{equation} \label{eqmax422}
\sup _{t \geq R} t \left|\left\{\tau \in(0, T) ;\|f(\tau)\|_{\dot{\mathcal{N}}_{q,\mu, \infty}^{s}}>t\right\}\right|^{\frac{1}{\alpha}}<2 \eta.
\end{equation}

Then we may define $g_{0}$ and $g_{1}$ by
\begin{equation}
g_{0}(t)=\left\{\begin{array}{ll}{\|f(t)\|_{\dot{\mathcal{N}}_{q,\mu, \infty}^{s}}} & {\text { for } t \in(0, T) \text { such that }\|f(t)\|_{\dot{\mathcal{N}}_{q,\mu, \infty}^{s}} \leq R} \\ {0} & {\text { for } t \in(0, T) \text { such that }\|f(t)\|_{\dot{\mathcal{N}}_{q,\mu, \infty}^{s}}>R}\end{array}\right.
\end{equation}
and $g_{1}(t)=\|f(t)\|_{\dot{\mathcal{N}}_{q,\mu,  \infty}^{q}}-g_{0}(t),$ respectively.
Obviously, we see that $g_{0} \in L^{\infty}(0, T).$  Since
$g_{1}(t)=0$ for all $t \in(0, T)$ such that $\|f(t)\|_{\dot{\mathcal{N}}_{q,\mu, \infty}^{s}} \leq R,$
it follows from (\ref{eqmax422}) that

\begin{eqnarray*}
\left\|g_{1}\right\|_{L^{\alpha, \infty}(0, T)}
&=&\sup _{t>0} t \left|\left\{\tau \in(0, T) ; g_{1}(\tau)>t\right\}\right|^{\frac{1}{\alpha}}\\
&=&\max \left\{\sup _{0<t \leq R} t \left|\left\{\tau \in(0, T) ; g_{1}(\tau)>t\right\}\right|^{\frac{1}{\alpha}}, \quad \sup _{R<t<\infty} t \left|\left\{\tau \in(0, T) ; g_{1}(\tau)>t\right\}\right|^{\frac{1}{\alpha}}\right\}\\
&=&\max \left\{R \left|\left\{\tau \in(0, T) ; g_{1}(\tau)>R\right\}\right|^{\frac{1}{\alpha}}, \quad \sup _{R<t<\infty} t \left|\left\{\tau \in(0, T) ; g_{1}(\tau)>t\right\}\right|^{\frac{1}{\alpha}}\right\}\\
&=&\sup _{R \leq t<\infty} t \left|\left\{\tau \in(0, T) ; g_{1}(\tau)>t\right\}\right|^{\frac{1}{\alpha}}
\leq 2 \eta
\end{eqnarray*}
which yields (\ref{eqmax414}).
Concerning validity of  (\ref{eqmax108}),
 the proof is parallel to that of (\ref{eqmax105}) as in the above case (i). So,
we may omit it.

Now it remains to prove uniqueness of solutions in the class (\ref{eqmax107}) under the condition (\ref{eqmax109}).
Suppose that $u_{1}$ and $u_{2}$ are two solutions of (\ref{eqmaxns}) on $(0, T_{*})$ in the class (\ref{eqmax107}) satisfying
\begin{eqnarray*}
&&\sup _{N \leq j<\infty} 2^{(-1+(n-\mu) / p) j}\left\|\varphi_{j} * a\right\|_{\mathcal{M}_{p,\mu}}\\
&&+\limsup _{t \rightarrow \infty} t \left|\left\{\tau \in\left(0, T_{*}\right) ; F_{i}(t) \equiv\left\|u_{i, t}(t)\right\|_{\dot{\mathcal{N}}_{q, \mu, \infty}^{s}}+\left\|A u_{i}(t)\right\|_{\dot{\mathcal{N}}_{q,\mu, \infty}^{s}}>t\right\}\right|^{\frac{1}{\alpha}}
 \leq \kappa
\end{eqnarray*}
for $i = 1, 2.$ In the same way as in (\ref{eqmax419}), we see that
\begin{eqnarray}\label{eqmax424}
&&\|e^{-t A} a\|_{L^{\alpha_{0}, \infty}(0, \tau ; \dot{\mathcal{N}}_{q_{0},\mu_{0}, \infty}^{s_{0}})}
+\|u_{i}\|_{X_{\tau}}\nonumber\\
 &&\leq C\left(\|a_{0}\|_{\mathcal{M}_{p,\mu}}\tau^{\frac{1}{2}(1-\frac{(n-\mu)}{p})}
 +\|g_{0}^{(i)}\|_{L^{\infty}(0, T_{*})} \tau^{\frac{1}{\alpha}}+\kappa\right),i=1,2
\end{eqnarray}
for all $\tau \in\left(0, T_{*}\right)$ with $C=C(n,\mu, q, \alpha)$ independent of $\tau,$ where $a_{0} \in \mathcal{M}_{p,\mu}$ is the same as in (\ref{eqmax413})
and where $g_{0}^{(i)} \in L^{\infty}(0, T_{*})$ is defined as in (\ref{eqmax414}) with $\|f(t)\|_{\dot{\mathcal{N}}_{q,\mu, \infty}^{s}}$ replaced by $F_{i}(t)$ for $i=1,2$.
Now we take
\begin{equation}
\kappa=\frac{1}{4 C}, \quad \tau=\frac{1}{\left(4 C\left(\left\|a_{0}\right\|_{\mathcal{M}_{p,\mu}}+\sum_{i=1}^{2}\left\|g_{0}^{(i)}\right\|_{L^{\infty}\left(0, T_{*}\right)}\right)\right)^{\frac{1}{\gamma}}}
\end{equation}
with $\gamma=\min \left\{\frac{1}{2}\left(1-\frac{(n-\mu)}{p}\right), \frac{1}{\alpha}\right\}$.
Then it follows from (\ref{eqmax410}) and (\ref{eqmax424}) that
\begin{equation}
\|w\|_{X_{\tau}} \leq \frac{1}{2}\|w\|_{X_{\tau}}
\end{equation}
which yields that $w(t) \equiv u_{1}(t)-u_{2}(t)=0$ for $t \in[0, \tau]$.
Repeating this argument on the interval $[\tau, 2 \tau],$ we have that $w(t)=0$ on $[0,2 \tau].$
Proceeding similarly beyond $t \geq 2 \tau$ within finitely many steps, we conclude that
$w(t)=0$ on $\left[0, T_{*}\right),$ which yields the desired uniqueness.
This
proves Theorem \ref{thm2}.

\subsection{Proof of Theorem \ref{thm3}}

Let $a \in \dot{\mathcal{N}}_{p,\mu, r}^{-1+\frac{(n-\mu)}{p}}$ and $f \in L^{\alpha, r}(0, \infty ; \dot{\mathcal{N}}_{q,\mu, \infty}^{s})$ for $2 / \alpha+(n-\mu) / q-s=3$ with
$1<q<\infty, 1<\alpha<\infty, 0\leq \mu<n$ and $-1<s$ satisfying
(\ref{eqmax102}). For such data $a$ and $f$ with the smallness condition as in (\ref{eqmax111}), we may construct the solution $v$ of (\ref{ns}) on $(0, \infty) .$ More precisely, in the same way as in Subsection $4.1,$ let us define the Banach space $X_{\infty}$ by
\begin{equation*}
X_{\infty} \equiv\left\{v : \mathbb{R}^{n} \times(0, \infty) \rightarrow \mathbb{R}^{n} ; v_{t}, A v \in L^{\alpha, r}(0, \infty ; \dot{\mathcal{N}}_{q,\mu, \infty}^{s}), v(0)=0\right\}
\end{equation*}
with the norm

\begin{equation*}
\|v\|_{X_{\infty}}=\left\|v_{t}\right\|_{L^{\alpha, r}(0, \infty ; \dot{\mathcal{N}}_{q,\mu, \infty}^{s})}+\|A v\|_{L^{\alpha, r}(0, \infty ; \dot{\mathcal{N}}_{q,\mu, \infty}^{s})}.
\end{equation*}
We need to find a solution $v$ of (\ref{ns1}) with $T_{*}=\infty$ in the class $X_{\infty}$ provided $a$ and $f$ satisfy
the condition (\ref{eqmax111}). Since Theorem \ref{thm1} and Lemmas  \ref{lemmax31} and \ref{lemmax32} hold even in the spaces
$L^{\alpha, r}(0, \infty ; \dot{\mathcal{N}}_{q,\mu, \infty}^{s})$ and $L^{\alpha_{0}, r}(0, \infty ; \dot{\mathcal{N}}_{q_{0}, \mu,\infty}^{s_{0}})$ and since all constants $C$ appearing in the estimates (\ref{eqmax103}), (\ref{eqmax301}), (\ref{eqmax302}) and (\ref{eqmax306}) can be chosen independently of $T,$ we see from (\ref{eqmax405}) that the solution $v$ of (\ref{ns1})  on $(0, \infty)$ can be obtained provided the condition

\begin{equation}\label{eqmax426}
\|e^{-t A} a\|_{L^{s_{0}, q}(0, \infty ; \dot{\mathcal{N}}_{q_{0}, \mu, 1}^{s_{0}})}+\|u_{0}\|_{X_{\infty}}<\frac{1}{4 C}
\end{equation}
is fulfilled with the same constant $C$ as in (\ref{eqmax405}). Since the estimates (\ref{eqmax401}) and (\ref{eqmax402}) hold even for $T=\infty$ with the constant $C=C(n,\mu, q, \alpha, p, r)$ independently of $T,$ by taking $\varepsilon_{*}$ in (\ref{eqmax111}) sufficiently small, we have that

\begin{equation}\label{eqmax427}
\left\|u_{0}\right\|_{X_{\infty}}<\frac{1}{8 C}.
\end{equation}
Obviously by (\ref{eqmax302}) with $k=-1+(n-\mu) / p,$ choice of small $\varepsilon_{*}$ in (\ref{eqmax111}) enables us to ensure that
\begin{equation}\label{eqmax428}
\left\|e^{-t A} a\right\|_{L^{s_{0}, r}(0, \infty ; \dot{\mathcal{N}}_{q_{0},\mu, 1}^{s_{0}})}<\frac{1}{8 C}.
\end{equation}
Now, from (\ref{eqmax427}) and (\ref{eqmax428}), we see that the condition (\ref{eqmax426}) can be achieved under the hypothesis
of (\ref{eqmax111}), which yields the solution $u$ of (\ref{ns}) on $(0, \infty)$ in the class (\ref{eqmax112}).
The additional property of $u$ as in (\ref{eqmax113}) and uniqueness assertion are both immediate
consequences of Theorem \ref{thm2}, and so we may omit their proofs. This proves Theorem \ref{thm3}.

\subsection{Proof of Theorem \ref{thm4}}
Since $a \in \dot{\mathcal{N}}_{p^{*},\mu^{*}, r^{*}}^{k^{*}}$ for $k^{*}=2+(n-\mu^{*}) / p^{*}-(2 / \alpha^{*}+(n-\mu^{*}) / q^{*}-s^{*})$ with $(n-\mu^{*}) / q^{*} \leq (n-\mu^{*}) / p^{*}<2 / \alpha^{*}+(n-\mu^{*}) / q^{*}$
and $f \in L^{\alpha^{*}, r^{*}}(0, \infty ; \dot{\mathcal{N}}_{q^{*},\mu^{*}, \infty}^{s^{*}})$,
it follows from Theorem \ref{thm1} that
\begin{equation}\label{eqmax437}
\frac{d u_{0}}{d t}, A u_{0} \in L^{\alpha^{*}, r^{*}}(0, \infty ; \dot{\mathcal{N}}_{q^{*}, \mu^{*}, \infty}^{s^{*}}).
\end{equation}
Hence it suffices to show that the solution $v$ of $(\ref{ns1})$ satisfies that

\begin{equation}\label{eqmax438}
\frac{d v}{d t}, A v \in L^{\alpha^{*}, r^{*}}(0, \infty ; \dot{\mathcal{N}}_{q^{*}, \mu^{*}, \infty}^{s^{*}}).
\end{equation}
For that purpose, we need to return to the approximating solutions $\left\{v_{j}\right\}_{j=0}^{\infty}$ of (\ref{ns1j}).
Let us take $\alpha_{0}, q_{0}, \mu$ and $s_{0}$ so that $\alpha<\alpha_{0}, q \leq q_{0}, 0\leq \mu^{*}\leq\mu<n,$
$\max \{s+1, (n-\mu) / p-1\}<s_{0}$ and $2 / \alpha_{0}+(n-\mu) / q_{0}-s_{0}=1.$
Since $v_{j} \in X_{\infty},$ in the same way as in (\ref{eqmax113}) we obtain from (\ref{eqmax406}) and Lemma \ref{lemmax31} that
$v_{j} \in L^{\alpha_{0}, r}(0, \infty ; \dot{\mathcal{N}}_{q_{0},\mu, \infty}^{s_{0}})$
with the estimate
\begin{equation}\label{eqmax439}
\left\|v_{j}\right\|_{L^{\alpha_{0}, r}(0, \infty ; \dot{\mathcal{N}}_{q_{0},\mu, \infty}^{s_{0}})} \leq C\left\|v_{j}\right\|_{X_{\infty}} \leq C K
\end{equation}
for all $j=1,2, \cdots$, where $C=C\left(n,\mu, q, \alpha, q_{0}, \alpha_{0}, p, r\right).$ Set

\begin{equation*}
Y \equiv\left\{v : \mathbb{R}^{n} \times(0, \infty) \rightarrow \mathbb{R}^{n} ; v_{t}, A v \in L^{\alpha^{*}, r^{*}}(0, \infty ; \dot{\mathcal{N}}_{q^{*},\mu^{*}, \infty}^{s^{*}}), v(0)=0\right\}
\end{equation*}
with the norm $\|\cdot\|_{Y}$ defined by
\begin{equation*}
\|v\|_{Y} \equiv\left\|v_{t}\right\|_{L^{\alpha^{*}, r^{*}}(0, \infty ; \dot{\mathcal{N}}_{q^{*}, \mu^{*}, \infty}^{s^{*}})}+\|A v\|_{L^{\alpha^{*}, r^{*}}(0, \infty ; \dot{\mathcal{N}}_{q^{*},\mu^{*}, \infty}^{s^{*}})}
\end{equation*}

Since $a \in \dot{\mathcal{N}}_{p^{*},\mu^{*}, r^{*}}^{k^{*}}$,
 we have by (\ref{eqmax437}) and Lemma \ref{lemmax33} that
$-P(u_{0} \cdot \nabla u_{0}) \in L^{\alpha^{*}, r^{*}}(0, \infty ; \dot{\mathcal{N}}_{q^{*}, \mu^{*},\infty}^{s^{*}})$ with the estimate
\begin{eqnarray}\label{eqmax440}
&&\|\mathbb{P}(u_{0} \cdot \nabla u_{0})\|_{L^{\alpha^{*}, r^{*}}(0, \infty ; \dot{\mathcal{N}}_{q^{*},\mu^{*}, \infty}^{s^{*}})}\nonumber\\
&&\leq
C\|u_{0}\|_{L^{\alpha_{0}, r}(0, \infty ; \dot{\mathcal{N}}_{q_{0},\mu, \infty}^{s_{0}})}
(\|a\|_{\dot{\mathcal{N}}_{p^{*},\mu^{*}, r^{*}}^{k^{*}}}+\|u_{0}\|_{Y})\nonumber\\
&&\leq
C(\|e^{-t A} a\|_{L^{\alpha_{0}, r}(0, \infty ; \dot{\mathcal{N}}_{q_{0}, \mu, 1}^{s_{0}})}+\|u_{0}\|_{X_{\infty}})(\|a\|_{\dot{\mathcal{N}}_{p^{*},\mu^{*}, r^{*}}^{k^{*}}}+\|u_{0}\|_{Y})
\end{eqnarray}
Hence it follows from Theorem \ref{thm1} that the solution
$v_{0}$ of (\ref{ns10})
has an additional regularity such
as $v_{0} \in Y.$
Assume that $v_{j} \in Y$. Since $\alpha^{*} \leq \alpha<\alpha_{0}$ and $-1<s^{*} \leq s<s_{0}-1$,
it follows from Lemma \ref{lemmax33}, (\ref{eqmax439}) and (\ref{eqmax440}) that
\begin{equation}
\mathbb{P}(u_{0} \cdot \nabla v_{j}+v_{j} \cdot \nabla u_{0}+v_{j} \cdot \nabla v_{j}+u_{0} \cdot \nabla u_{0}) \in L^{\alpha^{*}, r^{*}}(0, \infty ; \dot{\mathcal{N}}_{q^{*},\mu^{*}, \infty}^{s^{*}})
\end{equation}
with the estimate
\begin{eqnarray}\label{eqmax441}
&&\|\mathbb{P}(u_{0} \cdot \nabla v_{j}+v_{j} \cdot \nabla u_{0}+v_{j} \cdot \nabla v_{j}+u_{0} \cdot \nabla u_{0})\|_{L^{\alpha^{*}, r^{*}}(0, \infty ; \dot{\mathcal{N}}_{q^{*},\mu^{*}, \infty}^{s^{*}})}\nonumber\\
&\leq& C
(2\|u_{0}\|_{L^{\alpha_{0}, r}(0, \infty ; \dot{\mathcal{N}}_{q_{0}, \mu, \infty}^{s_{0}})}
+\|v_{j}\|_{L^{\alpha_{0}, r}(0, \infty ; \dot{\mathcal{N}}_{q_{0}, \mu, \infty}^{s_{0}})})\|v_{j}\|_{Y}\nonumber\\
&+&C(\|e^{-t A} a\|_{L^{\alpha_{0}, r}(0, \infty; \dot{\mathcal{N}}_{q_{0},\mu, 1}^{s_{0}})}+\|u_{0}\|_{X_{\infty}} )(\|a\|_{\dot{\mathcal{N}}_{p^{*}, \mu^{*}, r^{*}}^{k^{*}}}+\|u_{0}\|_{Y})\nonumber\\
&\leq& C_{*}
(2(\|e^{-t A} a\|_{L^{\alpha_{0}, r}(0, \infty ; \dot{\mathcal{N}}_{q_{0}, \mu, 1}^{s_{0}})}
+\|u_{0}\|_{X_{\infty}})+K)\|v_{j}\|_{Y}\nonumber\\
&+& C_{*}(\|e^{-t A} a\|_{L^{\alpha_{0}, r}(0, \infty; \dot{\mathcal{N}}_{q_{0}, \mu, 1}^{s_{0}} )}
+\|u_{0}\|_{X_{\infty}} )(\|a\|_{\dot{\mathcal{N}}_{p^{*}, \mu^{*}, r^{*}}^{k^{*}}}+\|u_{0}\|_{Y}).
\end{eqnarray}
where $C_{*}$ is the constant independent of $j = 1, 2, \cdots.$
Hence it follows from Theorem \ref{thm1} that the solution $v_{j+1}$ of
(\ref{ns1}) satisfies $v_{j+1} \in Y.$
By induction, it holds $v_{j} \in Y$ for all $j=1,2, \cdots$ and again from
(\ref{eqmax103}) and (\ref{eqmax441}) we obtain
\begin{eqnarray*}
&&\|v_{j+1}\|_{Y}\\
&&\leq C^{*}
\left(2(\|e^{-t A} a\|_{L^{\alpha_{0}, r}(0, \infty ; \dot{\mathcal{N}}_{q_{0}, \mu, 1}^{s_{0}})}+\|u_{0}\|_{X_{\infty}})+K\right)\|v_{j}\|_{Y}\\
&&+C_{*}(\|e^{-t A} a\|_{L^{\alpha_{0}, r}(0, \infty ; \dot{\mathcal{N}}_{q_{0}, \mu, 1}^{s_{0}})}+\|u_{0}\|_{X_{\infty}})(\|a\|_{\dot{\mathcal{N}}_{p^{*}, \mu^{*}, r^{*}}^{k^{*}}}+\|u_{0}\|_{Y}),
\end{eqnarray*}
where $C_{*}=C_{*}(n,\mu,\mu^{*}, q, \alpha, s, r,q^{*}, \alpha^{*}, s^{*}, r^{*})$.
Therefore, similarly to (\ref{eqmax407}), we have
\begin{eqnarray*}
\|w_{j+1}\|_{Y} \leq C_{*}\left(2(\|e^{-t A} a\|_{L^{\alpha_{0}, r}(0, \infty ; \dot{\mathcal{N}}_{q_{0}, \mu,1}^{s_{0}})}+\|u_{0}\|_{X_{\infty}})+2 K\right)\|w_{j}\|_{Y}
,j=0,1, \cdots.
\end{eqnarray*}
If
\begin{equation}\label{eqmax442}
\|e^{-t A} a\|_{L^{\alpha_{0}, r}(0, \infty ; \dot{\mathcal{N}}_{q_{0}, \mu, 1}^{s_{0}})}+\|u_{0}\|_{X_{\infty}}+K<\frac{1}{2 C_{*}},
\end{equation}
then it holds that
\begin{equation*}
\sum_{j=1}^{\infty}\left\|w_{j}\right\|_{Y}<\infty,
\end{equation*}

which implies that the limit $v$ of $\left\{v_{j}\right\}_{j=1}^{\infty}$ belongs to $Y$.

Now it remains to show (\ref{eqmax442}). Since $K$ can be taken arbitrarily small in accordance with
the size of $\|e^{-t A} a\|_{L^{\alpha_{0}, r}(0, \infty ; \dot{\mathcal{N}}_{q_{0}, \mu, 1}^{s_{0}})}+\|u_{0}\|_{X_{\infty}}$, it follows from (\ref{eqmax302}) and Theorem \ref{thm1} that there is a constant
 $\delta^{\prime}=\delta^{\prime}(n,\mu, \mu^{*}, q, \alpha, s, r, q^{*}, \alpha^{*}, s^{*}, r^{*}) \leq \delta$ such that if
\begin{equation*}
\|a\|_{\dot{\mathcal{N}}_{p,\mu, r}^{-1+(n-\mu) / p}}+\|f\|_{L^{\alpha, r}(0, \infty ; \dot{\mathcal{N}}_{q,\mu, \infty}^{s})} \leq \delta^{\prime}
\end{equation*}
then the condition (\ref{eqmax442}) is fulfilled. This completes the proof of Theorem \ref{thm4}.

\textbf{Acknowledgement}

The research of B. Guo
is partially supported by the National Natural Science Foundation
of China, grant 11731014.



\begin{thebibliography}{}



\bibitem{a1}
   H. Amann, \textit{Liner and Quasilinear Parabolic Problems, Volume I Abstract Linear Theory.} Monographs
in Mathematics 89 Birkh\"{a}user-Verlag, Basel-Boston-Berlin 1995.

\bibitem{adams}
D. R. Adams,
\textit{Morrey spaces}.
Lecture Notes in Applied and Numerical Harmonic Analysis,
 Birkh\"{a}user, 2015.


\bibitem{ap}
Marcelo F. de Almeida, Juliana C.P. Precioso,
Existence and symmetries of solutions in Besov-Morrey spaces
for a semilinear heat-wave type equation,
\textit{J. Math. Anal. Appl.} 432 (2015) 338-355.


\bibitem{a2}
     H. Amann, On the strong solvability of the Navier-Stokes equations,
     \textit{J. Math. Fluid Mech.}
     2 (2000) 16-98 .


%
%
%
%




\bibitem{bp}
   J. Bourgain,  N. Pavlovi\'{c},
    Ill-posedness of the Navier-Stokes equations in a critical space in 3D,
    \textit{ J. Func. Anal.} 255 (2008)  2233-2247 .





\bibitem{c}
M. Cannone, \textit{Ondelettes, Paraproduits et Navier-Stokes.} Diderot Editeur, Paris 1995.



\bibitem{cm}
  M. Cannone, Y. Meyer,
  Littlewood-Paley decompositions and Navier-Stokes equations,
 \textit{ Meth. Appl. Anal.}  2 (3), 307-319.

 \bibitem{ck}
 M. Cannone,  G. Karch,
  Smooth or singular solutions to the Navier-Stokes system,
 \textit{ J. Differential Equations} 197 (2004) 247-274.

 \bibitem{cp1}
 M. Cannone,  F. Planchon,
  Self-similar solutions for Navier-Stokes equations in $\mathbb{R}^{3}$,
  \textit{ Comm. Partial Differential Equations} 21 (1996) 179-194.

  \bibitem{cp2}
 M.  Cannone,  F. Planchon,
 On the nonstationary Navier-Stokes equations with an external force,
\textit{Adv. Differential Equations} 4 (1999) 697-730.

%

\bibitem{fs1}
R. Farwig,   H. Sohr,
Optimal initial value conditions for the existence of local strong solutions of
the Navier-Stokes equations,
 \textit{Math. Ann.} 345  (2009)  631-642.



\bibitem{fsv}
R. Farwig,  H. Sohr,  W. Varnhorn,
Optimal initial value conditions for local strong solutions of
the Navier-Stokes equations,
\textit{ Ann. Univ. Ferrara} 55  (2009) 89-110.

\bibitem{fk}
H. Fujita,  T. Kato,
On the nonstationary Navier-Stokes system,
\textit{Arch. Rational Mech. Anal.}
16 (1964) 269-315.



\bibitem{gm}
Y.  Giga, T. Miyakawa,
 Solutions in $L^{r}$ of the Navier-Stokes initial value problem,
\textit{ Arch. Rational Mech. Anal.} 89 (1985) 267-281.

\bibitem{gs}
 Y. Giga,  H.  Sohr,
 Abstract $L^{p}$ estimates for the Cauchy problem with applications to the
  Navier-Stokes equations in exterior domains,
\textit{ J. Funct. Anal.} 102 (1991) 72-94.


\bibitem{k1}
K. Kashiwagi,
Well-posedness of the Navier-Stokes equations with external force in $L^{p}$ space
(Japanese). Master thesis of Department of Mathematics, Graduate School of Science, Shizuoka
University 2015.

\bibitem{k2}
 T. Kato, Strong $L^{p}$-solution of the Navier-Stokes equation in $\mathbb{R}^{m}$, with applications to weak solutions,
\textit{Math. Z.} 187 (1984) 471-480.

\bibitem{hw12}
T. Kato,
Strong solutions of the Navier-Stokes equation in Morrey spaces,
 \textit{Bull. Braz. Math. Soc. (N.S.)} 22 (1992) 127-155.




\bibitem{kt}
 H. Koch, D. Tataru,
  Well-posedness for the Navier-Stokes equations,
 \textit{ Adv. Math. } 157 (2001) 22-35.






\bibitem{ks1}
 H.  Kozono, Y. Shimada,
 Bilnear estimates in homogeneous Triebel-Lizorkoin spaces and the
  Navier-Stokes equations,
 \textit{Math. Nachr.} 276 (2004) 63-74.


 \bibitem{ks2}
 H.  Kozono,  S. Shimizu,
 Navier-Stokes equations with external forces in Lorentz spaces and
its application to the self-similar solutions,
\textit{J. Math. Anal. Appl.} 458 (2018)  1693-1708.

\bibitem{ks3}
 H.  Kozono, S. Shimizu,
 Navier-Stokes equations with external forces in time-weighted
Besov spaces, \textit{Math. Nachr.} 291 (2018) 1781-1800.


\bibitem{ks6}
H. Kozono, S. Shimizu,
Strong solutions of the Navier-Stokes equations based on the maximal Lorentz
regularity theorem in Besov spaces,
\textit{J. Funct. Anal.} 276 (2019) 896-931.


\bibitem{ks5}
H.  Kozono,  H. Sohr,
Regularity criterion on weak solutions to the Navier-Stokes equations,
\textit{Advances in Differential Equations} 2  (1997) 535-554.






\bibitem{ky1}
H.  Kozono,  M. Yamazaki,
Semilinear heat equations and the Navier-Stokes equation with
distributions in new function spaces as initial data,
\textit{Comm. Partial Differential Equations}
19 (1994) 959-1014.


\bibitem{ky2}
H. Kozono,  M. Yamazaki,
Local and global unique solvability of the Navier-Stokes exterior
problem with Cauchy data in the space $L^{n, \infty}$,
\textit{ Houston J. Math. } 21 (1995) 755-799.


\bibitem{hw17}
 H. Kozono, M. Yamazaki,
  The stability of small stationary solutions in Morrey spaces of the Navier-Stokes equation,
\textit{Indiana Univ. Math. J.} 44  (1995) 1307-1336.







\bibitem{hw23}
T. Miyakawa,
 On Morrey spaces of measures: basic properties and potential estimates,
 \textit{Hiroshima Math. J.}  20  (1990) 213-222.



\bibitem{n}
Y. Nakamura,
Well-posedness of the Navier-Stokes equations with external force in homogeneous
Besov space (Japanese). Master thesis of Department of Mathematics, Graduate
School of Science, Shizuoka University 2015.

\bibitem{p}
F. Planchon,
Asymptotic behavior of global solutions to the Navier-Stokes equations in $\mathbb{R}^{3}$,
\textit{Rev. Mat. Iberoamericana} 14 (2001) 2211-2226.

\bibitem{hw25}
J. Peetre,
 On the theory of $\mathcal{L}_{p, \lambda}$ spaces,
 \textit{J. Funct. Anal.} 4 (1969) 71-87.



\bibitem{rs}
M. Reed,   B. Simon,
 \textit{Method of Modern Mathematical Physics II: Fourier Analysis,
Self-Adjointness.} Academic Press, San Diego-New York-Berkeley-Boston-London-Sydney-
Tokyo-Tronto 1975


\bibitem{t}
M. Taylor,
 Analysis on Morrey spaces and applications to Navier-Stokes and other evolution equations,
 \textit{ Comm. Partial Differential Equations} 17 (1992) 1407-1456.

\bibitem{w}
B.  Wang,
 Ill-posedness for the Navier-Stokes equations in critical Besov spaces
 $\dot{B}_{\infty, q}^{-1}$,
 \textit{Adv. Math.} 268 (2015) 350-372.

 \bibitem{y}
 T. Yoneda,
  Ill-posedness of the 3D-Navier-Stokes equations in Besov spaces near $BMO^{-1}$,
\textit{J. Funct. Anal.} 258 (2010) 3376-3387.













\end{thebibliography}
\end{document}